\documentclass[a4paper]{scrartcl} %{etna}
\usepackage[utf8]{inputenc}
\usepackage[T1]{fontenc}
\usepackage[english]{babel}
\usepackage{amsmath,amssymb,amsfonts,siunitx,commath}
\usepackage{amsthm}
\usepackage{tikz,url}
\usepackage{geometry}
\usepackage{mathtools}
\mathtoolsset{showonlyrefs}
\usepackage{subcaption}
\usepackage{algorithm}
\usepackage{algpseudocode}
\usepackage{hyperref}
\usepackage{enumerate}
\usepackage{listings}
\usepackage[normalem]{ulem}

\DeclareMathOperator*{\argmax}{arg\,max}
\DeclareMathOperator*{\argmin}{arg\,min}

\newcommand{\R}{\mathbb{R}}
\newcommand{\E}{\mathbb{E}}

\newcommand{\Z}{\mathbb{Z}}
\newcommand{\T}{\mathbb{T}}

\newcommand{\N}{\mathbb{N}}

\newcommand\prox{\mathrm{prox}}

\newcommand\SPD{\mathrm{SPD}}

\newcommand{\tT}{\mathrm{T}}
\newcommand{\wN}{\mathcal{N}_w}
\newcommand{\diag}{\mathrm{diag}}

\newcommand{\n}{\phantom{0}}
\newcommand{\lin}{\mathrm{span}}
\theoremstyle{plain}
\newtheorem{lemma}{Lemma}[section]
\newtheorem{theorem}[lemma]{Theorem}

\newtheorem{proposition}[lemma]{Proposition}
\newtheorem{remark}[lemma]{Remark}
\theoremstyle{definition}

\newtheorem{example}[lemma]{Example}

\usepackage{lineno} 
\title{Sparse Mixture Models inspired by ANOVA  Decompositions
}

\author{Johannes Hertrich\footnotemark[1]
	\and
	Fatima Antarou Ba\footnotemark[1]
	\and
	Gabriele Steidl\footnotemark[1]
}

\begin{document}
	\date{\today}
	\maketitle
	
	\footnotetext[1]{
		TU Berlin,
		Stra{\ss}e des 17. Juni 136, 
		D-10587 Berlin, Germany,
		\{j.hertrich, fatimaba, steidl\}@math.tu-berlin.de.
	} 
	
%	\setbibdata{1}{xx}{46}{2021}
	\begin{abstract}
		Inspired by the analysis of variance (ANOVA) decomposition of functions 
		we propose a Gaussian-Uniform mixture model on the high-dimensional torus
		which relies on the assumption that the function 
		we wish to approximate can be well explained by limited variable interactions.
		We consider three approaches, namely
		wrapped Gaussians, diagonal wrapped Gaussians and products of
		von Mises distributions.
		The sparsity of the mixture model is ensured by the fact that its summands 
		are 
		products of Gaussian-like density functions acting on low dimensional spaces
		and uniform probability densities defined on the remaining directions.
		To learn such a sparse mixture model from given samples, we propose
		an objective function consisting of the negative log-likelihood function of the mixture model
		and a regularizer that penalizes the number of its summands.
		For minimizing this functional we combine the Expectation Maximization algorithm with a proximal step 
		that takes the regularizer into account. 
		To decide which summands of the mixture model
		are important, we apply a Kolmogorov-Smirnov test.
		Numerical examples demonstrate the performance of our approach.
	\end{abstract}
%	\begin{keywords}
%		sparse mixture models, ANOVA decomposition, wrapped Gaussian distribution, von Mises distribution,  approximation of high dimensional probability density functions, Kolmogorov-Smirnov test.
%	\end{keywords}
	%------------------------------------------
	\section{Introduction} \label{sec:intro}
	%------------------------------------------
	Most high-dimensional real-world systems are dominated by a small number of low-complexity interactions
	\cite{WH2011}.
	This is the background of extensive research to represent functions acting on high-dimensional data
	by functions defined on lower dimensional spaces. Approaches include
	active subspace methods  \cite{CDW2014,CEHW2017,FSV2012} 
	and random features \cite{CJJ2012,HSSTTW2021,LTOS2019,RR2008,YLMJZ2012}. 
	
	This paper was inspired by the analysis of variance (ANOVA) decomposition
	of functions \cite{CMO1997,Gu2013,Holtz2011,KSWW2010,LO2006} 
	which decomposes a function uniquely into the sum of functions depending on the different 
	variable combinations.
	In practice it can often be assumed that the significant part of a functions 
	can be explained by the simultaneous interactions of only a small number of variables,
	which is also in the spirit of \cite{DPW2011}.
	The amazing result of  Potts and Schmischke in \cite{BPS2020,PS2021a,PS2021b} 
	show that high-dimensional functions with a sparse ANOVA decomposition can be reconstructed using
	their approximation in the Fourier domain by rather few samples $t^i \in \mathbb T^d$ and $f(t^i) \in \mathbb R$,
	$i=1,\ldots,N$. 
	While the theory relies mainly on uniformly sampled points on the high-dimensional torus $\T^d$ or on
	$[0,1]^d$, also real-world data sets can be approximated in a way that beats
	state-of-the-art methods as the gradient boosting machine \cite{GPT2020},
	random forest \cite{GPT2020}, sparse random features \cite{HSSTTW2021} 
	and local learning regression neural networks \cite{KM2015}.
	
	In this paper, we assume that we are given high-dimensional samples $(x^i)_i$ 
	from a distribution with unknown 
	probability density function $f$ rather than interpolation knots $(t^i, f(t^i))_i$.  
	Since e.g.\ attribute ranking can be used to remove unimportant variables immediately and to reduce the dimensionality of the problem,
	we concentrate on functions, where each variable has an influence, but not simultaneously with all other ones.
	Then, we are interested in mixture models with sparse components in the sense that
	they depend only on the data in smaller dimensions.
	We propose to learn such a mixture model by minimizing a penalized negative log-likelihhod function
	in connection with a Kolmogorov-Smirnov test to find the active variables in the summands of the mixture model.
	Once a mixture model is fitted, a natural way to identify the influence of attributes to the outcome is given by
	adjusting samples to the appropriate summands of the mixture.
	
	Our model appears to be opposite to recently introduced 
	mixture models which components rely on projections into sparse subspaces of the high-dimensional data space
	as mixtures of probabilistic PCAs (MPPCA) \cite{TB1999}, 
	high-dimensional data clustering (HDDC) \cite{BGS2006}, 
	high-dimensional mixture models for unsupervised image denoising (HDMI) \cite{HBD2018}, and PCA-GMMs \cite{HNABBSS2020}. 
	For more information, see Remark \ref{rem:opposite}.
	
	The paper is organized as follows: 
	In Section \ref{sec:anova}, we first recall the sparse ANOVA decomposition on the $d$-dimensional torus.
	Then we introduce appropriate sparse mixture models having  components which are products of a
	Gaussian-like density function on the $n$-dimensional torus ($n \ll d$) and a the uniform density
	on the $d-n$-dimensional torus. 
	We propose three Gaussian-like settings, namely the wrapped normal distribution,
	the diagonal wrapped normal distribution, and products of von Mises distributions.
	Further, we discuss the ANOVA decomposition of the mixture models 
	based on the notation of identifiable parameterized families of functions.
	Section \ref{sec:spamm} deals with the learning of the sparse mixture model.
	Based on an objective function consisting of the negative log-likelihood function penalized
	by a sparsity term for the number of coefficients, we propose to
	apply an Expectation Maximization (EM) algorithm in combination 
	with a proximal step and a Kolmogorov-Smirnov test.
	Section \ref{sec:numerics} demonstrates the performance of our model by several examples.
	{The code is available online\footnote{\url{https://github.com/johertrich/Sparse_Mixture_Models}}.}
	%Conclusions are drawn and further directions of research are given in Section \ref{sec:conclusions}.
	The Appendix \ref{sec:A} summarizes the EM algorithms for the tree Gaussian-like mixture models,
	and Appendix \ref{sec:test} briefly shows how the Kolmogorov-Smirnov test works.
	
	%------------------------------------------
	\section{ANOVA Decomposition and Mixture Models} \label{sec:anova}
	%------------------------------------------
	Let $[d]\coloneqq\{1,...,d\}$ with the convention that $[0]=\emptyset$, and let $\mathcal P([d])$
	be the power set of $[d]$.
	Further, for $u \subseteq [d]$, we write $u^c\coloneqq [d]\backslash u$ and $x_u\coloneqq (x_i)_{i\in u}$.
	By $\T^d=\R^d/\Z^d=[0,1[^d$, we denote the $d$-dimensional torus and by $I_d$ the $d \times d$ identity matrix.
	
	We are interested in the additive decomposition of integrable functions  $f:\T^d \rightarrow \R$  into lower dimensional components 
	\begin{align}\label{eq_spam}
		f(x)=\sum_{u\subseteq [d]} f_u(x_u), \quad f_u\colon \T^{|u|}\to\R.
	\end{align}
	In general, such decomposition is not unique.
	We rely on two special decompositions, namely the ANOVA decomposition which was the motivation of this work
	and sparse mixture models. In the following we introduce both concepts and explain their relation.
	
	\paragraph{ANOVA decomposition}
	For any integrable function $f:\T^d \rightarrow \R$, there exists
	a unique decomposition of the form \eqref{eq_spam},
	the so-called \emph{analysis of variance (ANOVA) decomposition} determined by
	\begin{align} \label{eq:impo}
		f_u  \coloneqq (P_u f) - \sum_{v\varsubsetneq  u} f_{v}
		= \sum_{v \subseteq u} (-1)^{|u| - |v|} P_v f,
	\end{align}
	where
	\begin{equation} \label{eq:proj}
		(P_u f)(x) \coloneqq \int_{\T^{d-|u|} } f(x) \, d x_{u^c}.
	\end{equation} 
	
	The following proposition recalls that the ANOVA decomposition of a function which is the sum of
	lower dimensional functions can only contain summands acting on the same subspaces.
	
	\begin{proposition}\label{lem:ano}
		Let $W \subseteq \mathcal P([d])$. Then, a function $f$ of the form
		\begin{equation} \label{eins}
			f(x) = \sum_{w \in W} g_w (x_w)
		\end{equation}
		has an ANOVA decomposition of the form
		$$
		f = \sum_{u \in \bar W}  f_u,
		$$
		where $\bar W$ denotes the set $\{u\subseteq w: w\in W\}$.
	\end{proposition}
	
	\begin{proof} Let $T_u$ denote the linear operator which maps a function $f$ to its ANOVA component $f_u$.
		Then we have for the function in \eqref{eins} that
		$$
		f_u = T_u f = \sum_{w \in W} T_u g_w.
		$$
		and it remains to show that for $w$ not containing $u$ it holds $T_u g_w = 0$. Let $u$ be not contained in $w$.
		Using \eqref{eq:impo} and the facts that $g_w=P_w g_w$ and $P_vP_u=P_{v\cap u}$, 
		we obtain
		\begin{align}
			T_u g_w &=\sum_{v \subseteq u} (-1)^{|u| - |v|} P_v g_w
			=\sum_{v \subseteq u} (-1)^{|u| - |v|} P_{v}P_{w} g_w
			= \sum_{v \subseteq u} (-1)^{|u| - |v|} P_{v \cap w} g_w.
		\end{align}
		If $v\cap w = \emptyset$, the assertion follows since $\sum_{v \subseteq u} (-1)^{|u| - |v|} = 0$.
		Otherwise, we get with $n\coloneqq |u\setminus w|>0$ that
		\begin{align}
			T_u g_w &=\sum_{v_1 \subseteq u\cap w}\sum_{v_2\subseteq u\setminus w} (-1)^{|u| - (|v_1|+ |v_2|)} P_{v_1} g_w\\
			&=\sum_{v_1 \subseteq u\cap w} (-1)^{|u| - |v_1|} P_{v_1} g_w \sum_{v_2\subseteq u\setminus w}(-1)^{|v_2|}\\
			&=\sum_{v_1 \subseteq u\cap w} (-1)^{|u| - |v_1|} P_{v_1} g_w \sum_{j=0}^{n} \binom{n}{j}(-1)^{j} = 0,
		\end{align}
		which gives the assertion.
	\end{proof}
	
	In real-world applications, it often appears that the decomposition \eqref{eq_spam}
	does not contain all subsets of $[d]$, but just a smaller amount of subsets 
	$U \subset \mathcal P([d])$ which have cardinality not larger than some $n \ll d$
	or that $f$ can be at least well approximated by a sparse ANOVA decomposition
	\begin{align}\label{eq_sparse}
		\sum_{u\subseteq U} f_u(x_u).
	\end{align}
	Several authors examined the reconstruction of functions 
	having such a sparse ANOVA approximation from values 
	$\left( t^i,f(t^i) \right)$, $i=1,\ldots,N$
	of $f$.
	The setting in this paper is different. 
	
	\begin{remark}[Setting of this paper] \label{rem:setting}
		We deal with non-negative functions 
		$f\colon \mathbb T^d \rightarrow \R$ on the $d$-dimensional torus $\mathbb T^d$ fulfilling
		$
		\|f\|_{L_1(\mathbb T^d)} = 1
		$
		and consider them as probability density functions of a certain random variable
		$X\colon \Omega \rightarrow \mathbb T^d$.
		Instead of sampled function values, we assume that we are
		given samples $x^i$, $i=1,\ldots,N$ of the distribution with density $f$, i.e.\ realizations of the
		random variable $X$. This means that in contrast to the $t^i$, the samples $x^i$ inherit the properties of $f$.
		If the $t^i$, $i=1,\ldots,N$ are uniformly sampled, then clearly $f(t^i)$ times $t^i$ may serve as samples,
		i.e., the $t^i$ must be weighted with the values $f(t^i)$.
	\end{remark}
	
	\paragraph{Sparse Mixture Models}	
	In this paper, we aim to find an approximation of $f \in L_1(\mathbb T^d)$ by a mixture model 
	from samples of the corresponding distribution. 
	For an introduction to mixture models we refer to \cite{MP2000}.
	To this end, let $\Delta_K \coloneqq \{\alpha \in \mathbb R^K_{\ge 0} : \alpha^\tT 1_K = 1\}$
	with the vector $1_K$ consisting of $K$ entries 1, be the probability simplex, and
	$\SPD(d)$ the cone of symmetric, positive definite matrices.
	Assume that $f$ can be approximated by
	\emph{mixture models} of the form
	\begin{equation} \label{eq:sMM_model}
		p(x|\alpha,\vartheta) =\sum_{k=1}^K \alpha_k p_{u_k} (x_{u_k}|\vartheta_k), 
	\end{equation}
	where $u_k \in U \subset \mathcal P([d])$,
	$\alpha = (\alpha_k)_{k=1}^K \in \Delta_K$,  $\vartheta = (\vartheta_k)_{k=1}^K$ and 
	$p_{u_k}$ is a probability density functions on $\T^{n}$,  $n= |u_k|$. Note that the index sets $u_k$ are in general not pairwise different, i.e. $u_k = u_l$ can appear for $k \not = l$.
	If $U$ contains only sets of small cardinality, we call \eqref{eq:sMM_model} a \emph{sparse mixture model}.
	Indeed, the density $p_{u}$ determines the distribution 
	of a $\T^d$-valued random variable $X=(X_u,X_{u^c})$ characterized by
	$$
	(X_u,X_{u^c})\sim p_u(\cdot|\vartheta)\times \mathcal U_{\T^{d-|u|}}.
	$$
	This class of distributions includes for $u=\emptyset$ the uniform distribution on $\T^d$. 
	
	In this paper, we need the (absolutely continuous) \emph{normal or Gaussian distribution} on $\mathbb R^n$
	having the density function
	\begin{equation}\label{density_normal}
		\mathcal N(x|\mu,\Sigma) = (2\pi)^{-\frac{n }{2}} \abs{\Sigma}^{-\frac{1}{2}} 
		\,\exp\left(-\frac{1}{2}(x-\mu)^\tT \Sigma^{-1}(x-\mu) \right)
	\end{equation}  
	with mean $\mu\in \R^n$ and $\Sigma\in \SPD(n)$. This distribution has many characterizing properties
	which unfortunately  cannot be transferred to a ,,normal distribution'' on manifolds, see, e.g. \cite{Laus2019}.
	In this paper, we restrict our attention to the normal-like distributions $p_{u_k}$ on the $n = |u_k|$ dimensional $\T^n$ listed in the following example.
	
	\begin{example} \label{ex:distr}
		We focus on mixture models on $\T^d$ with low dimensional components 
		from one of the following distributions on $\T^n$, $n \ll d$:
		\begin{itemize}
			\item[i)] the \emph{wrapped normal distribution} 
			$$
			p_{G}(x|\mu,\Sigma)=\sum_{l\in\Z^{n}}\mathcal N(x+l|\mu,\Sigma) = \mathcal N_w(x|\mu,\Sigma),
			$$
			where $\mu \in \T^n$, $\Sigma \in \SPD(n)$. Note that $\wN(\mu,\Sigma)$ is characterized by the distribution of 
			$X-\lfloor X\rfloor$, where $X\sim\mathcal N(\mu,\Sigma)$.
			This formula allows us, to draw easily samples from $\wN(\mu,\Sigma)$.
			\item[ii)] the \emph{diagonal wrapped normal distribution}
			$$
			p_{dG}(x|\mu,\sigma^2)=p_{G}(x|\mu,\diag(\sigma^2)) = \sum_{l\in\Z^{n}} \prod_{j=1}^n \mathcal N (x_j+l_j|\mu_j,\sigma_j^2) 
			= 
			\prod_{j=1}^n \wN(x_j|\mu_j,\sigma_j^2),
			$$
			where ($\mathcal N_w$) $\mathcal N$ is the univariate (wrapped) Gaussian density function and $\sigma^2\in\R^n_{>0}$.
			\item[iii)] the \emph{von Mises distribution} on $\T^n$ with parameters $\mu \in \mathbb \T^n$ 
			and $\kappa \in \mathbb R^n_{>0}$ is the restriction of the probability density function of an isotropic 
			normal distribution to the unit circle and has the probability density function
			\begin{align} \label{eq_vM_density}
				p_M (x|\mu,\kappa)= \prod_{j=1}^n \tfrac{1}{I_0(\kappa_j)}\exp \left(\kappa_j \cos \left(2\pi(x_j-\mu_j) \right) \right) ,
			\end{align}
			where $I_0$ is the \emph{modified Bessel function of first kind of order $0$}. \hfill $\Box$
		\end{itemize}
	\end{example}
	
	The wrapped normal distribution inherits by definition several properties of the normal distribution in $\R^n$. 
	For example, we obtain directly 
	for independent $X\sim\wN(\mu,\Sigma)$, $Y\sim\wN(\mu',\Sigma')$ that $X+Y\sim\wN(\mu+\mu',\Sigma+\Sigma')$. 
	Similarly, we get that any marginal of $X$ is again a wrapped normal distribution.
	Other properties of the normal distribution are not transferred to the wrapped case. 
	For example, it holds on a circle that the von Mises distribution maximizes 
	the entropy and not the wrapped normal distribution, see \cite{JS2001}. 
	Indeed the von Mises distribution with parameters $(\mu,\kappa)$ is very similar to the one-dimensional 
	wrapped normal distribution with parameters $(\mu,\sigma^2)$, where the parameters are related via
	$\frac{I_1(\kappa)}{I_0(\kappa)}=\exp(-\tfrac{(2\pi)^2\sigma^2}{2})$, see \cite{K1974}.
	Thus, the von Mises distribution is often used instead of the wrapped normal distributions 
	with the benefit of a reduced complexity for evaluating the density function and estimating the parameters, 
	see e.g.\ \cite{B1989,F1987,KGH2014}.
	Unfortunately, there is no multivariate counterpart for this approximation.
	Finally, we like to mention that there also exist extensions 
	of the von Mises distribution to the (non tensor) multivariate case on $\T^d$, see \cite{M2010,MHTS2008,MTS2007}. 
	Unfortunately, the normalization constants of these multivariate von Mises distributions 
	have in general no closed form and the numerical approximation is very expansive.

	The following remark highlights the difference of our approach to another kind of ,,sparse''
	mixture models in the literature.
	
	\begin{remark} [An ,,opposite'' sparse mixture model]\label{rem:opposite}
		In a Gaussian setting, our sparse mixture model replaces the 
		inverse covariance matrices $\Sigma^{-1}$ in the summands of the mixture model
		by special matrices of low rank $|u|$. 
		This is opposite to replacing the covariance matrices $\Sigma$ themselves by low rank matrices
		as done in, e.g., \cite{GBKT2020,STA2021}. In other words, our paper addresses sparsity in the time domain, while
		the other authors consider the Fourier domain.
		
		In \cite{HNABBSS2020}, see also \cite{BGS2006,HBD2018},
		the authors considered so-called PCA-GMMs. 
		These are Gaussian mixture models, 
		where, up to a rotations, the summand densities are associated with random variables distributed as
		$$
		(X_u,X_{u^c})\sim \mathcal N(\mu,\Sigma)\times \mathcal N(0,\sigma^2 I)
		$$
		where $\sigma^2>0$ is a \emph{small}    fixed parameter which may account for noise in the data.
		Now, wrapping $X$ around the $d$-dimensional torus yields that $Y=X-\lfloor X\rfloor$ is distributed as
		$$
		(Y_u,Y_{u^c})\sim \wN(\mu,\Sigma)\times\wN(0,\sigma^2I).
		$$
		For $\sigma^2\to\infty$ this distribution converges  to
		$$
		\wN(\mu,\Sigma)\times \mathcal U_{\T^{d-|u|}},
		$$
		which is exactly how the components of our model \eqref{eq:sMM_model} are defined.
		Thus, in contrast to the PCA-GMM model, we have $\sigma^2\to\infty$ instead of a small or vanishing $\sigma$.
		\hfill $\Box$
	\end{remark}
	
	%----------------------
	Finally, we are interested in the ANOVA decomposition of our mixture models.
	To this end, we consider functions $h_u(\cdot|\vartheta)\colon\T^{|u|}\to\R$ 
	depending on $u\subseteq[d]$ 
	and $\vartheta\in\Theta_u$ 
	for some parameter space $\Theta_u$.
	We say that a family of probability density functions
	$\mathcal H= \{ h_u(\cdot|\vartheta): u\subseteq[d],\, \vartheta\in\Theta_u\}$ 
	is \emph{closed under projections}, 
	if for any $u,v\in[d]$, $\vartheta\in\Theta_u$ 
	there exists $\tilde\vartheta\in\Theta_{v\cap u}$ such that
	$$
	h_{v\cap u}(x_{v\cap u}|\tilde \vartheta)=P_{v\cap u} h_u(\cdot|\vartheta).
	$$
	In other words, marginals of functions in $\mathcal H$ have the same form.
	As already mentioned the family of wrapped Gaussians is closed under projection.
	Clearly this holds also true for families of direct products of univariate distributions.
	
	Further, the family $\mathcal H$ is called \emph{identifiable}, 
	if its elements are linearly independent in the vector space of all functions on $\T^d$, 
	i.e.\ for all $K \in \mathbb N$ and $\alpha_1,...,\alpha_K\in\R$ 
	it holds that 
	$\sum_{k=1}^K\alpha_k h_{u_k}(x_{u_k}|
	\vartheta_{u_k})=0$ implies $\alpha_k=0$ for all $k =1,\ldots,K$, see \cite{Teicher1961,YS1968}.
	It is known that the multivariate Gaussian family on $\mathbb R^d$ is identifiable \cite{DD2020,YS1968}.
	Further the univariate wrapped normal distribution \cite{HMS2004}  and the von Mises distribution 
	on $\T$ are identifiable \cite{FHW1981}. By \cite{Teicher1967}, also diagonal wrapped normal distributions in ii) and
	the products of von Mises distributions in iii) are identifiable. If the wrapped normal distribution on $\T^d$ 
	in i) is identifiable appears to be an open problem.
	
	Then we have the following proposition on the ANOVA decomposition of mixture models.
	
	\begin{proposition}\label{lem:anoo}
		Let $W \subseteq \mathcal P([d])$ and 
		$\mathcal H= \{ h_u(\cdot|\vartheta): u\subseteq[d],\, \vartheta\in\Theta_u\}$ 
		be an identifiable family of probability density functions 
		which is  closed under projections. 
		Further, let $g_w$, $w \in W$ be the linear combination of functions from $\{ h_w(\cdot|\tilde{\vartheta}_l): \tilde{\vartheta}_l \in \Theta_w\}$
		with positive coefficients. 
		Then, a function $f$ of the form
		$$
		f(x) = \sum_{w \in W} g_w(x_w)
		$$
		has the ANOVA decomposition
		$$
		f = \sum_{u \in \bar W} f_u
		$$
		with $f_u \neq 0$ for all $u \in \bar W$, where $\bar W$ denotes the set $\{v\subseteq w: w\in W\}$.             
	\end{proposition}
	
	\begin{proof}
		By Proposition \ref{lem:ano} we know already that
		$$
		f  = \sum_{u \in \bar W} f_u 
		$$
		so that it remains to show that none of these summands vanishes.
		Assume in contrary that there exists
		$u \in \bar W$ such that $f_u=0$.
		By \eqref{eq:impo} we have 
		\begin{align}\label{to_zero}
			0=f_u= P_u f + \sum_{v \subsetneq u} (-1)^{|u| - |v|} P_v f.
		\end{align}
		Since $\mathcal H$ is closed under projection and the $g_w$ are positive linear combinations 
		of functions from $\{h_w(\cdot|\vartheta): \vartheta \in \Theta_w\}$, we have $g_w=\sum_{l=1}^{K_l} \alpha_{w,l}h_w(\cdot|\tilde\vartheta_l)$, $\alpha_{w,l}>0$ and therefore
		\begin{align}
			P_u f&=\sum_{w\in W} P_u g_w 
			=\sum_{w\in W,u\subseteq w} P_u g_w + \sum_{w\in W, u\not\subseteq w} P_u g_w\\
			&= \sum_{w\in W,u\subseteq w} \sum_{l=1}^{K_l} \alpha_{w,l} P_u (h_w(\cdot,\tilde\vartheta_l)) + f_1\\
			&=\sum_{k=1}^{K} \alpha_k h_u(\cdot|\vartheta_k) + f_1
		\end{align}
		for some $K\in\N$, positive coefficients $\alpha_k\in\R_{>0}$, 
		$f_1\in\lin\{h_v(\cdot|\vartheta):v\subsetneq u, \, \vartheta \in \Theta_v\}$ 
		and pairwise distinct $\vartheta_k$, $k=1,...,K$.
		Since $u \in \bar W$, we have that $K>0$.% there exists $k\in\{1,...,K\}$ with $\alpha_k>0$.
		Further, it holds for $v\subsetneq u$ that
		$$
		P_v f=\sum_{w\in W}P_v g_w=f_2 \in \lin\{ h_{\tilde v}(\cdot|\vartheta): 
		\tilde v\subsetneq u, \vartheta \in \Theta_{\tilde v} \}.
		$$
		Putting the last two formulas together, we obtain in \eqref{to_zero} that
		$$
		0=\sum_{k=1}^{K} \alpha_k  h_u(\cdot|\vartheta_k)+f_3
		$$
		where $f_3 \in\lin\{h_v(\cdot|\vartheta):v\subsetneq u, \, \vartheta \in \Theta_v\}$. 
		Now the identifiability of $\mathcal H$ yields that $\alpha_k=0$ for all $k=1,...,K$, which is a contradiction.
	\end{proof}
	
	%---------------------------------------------------------------
	\section{Learning Sparse Mixture Models} \label{sec:spamm}
	%---------------------------------------------------------------
	Our approach for learning a sparse mixture model consists of three
	items. First, we need a rough approximation of the involved index sets $u_k$ 
	which is done in Subsection \ref{sec:find_u}.
	Then we consider an objective function consisting of
	the log-likelihood of the corresponding mixture model and an additional term that penalizes too many summands and supports further sparsity of the mixture model.
	To minimize this objective function we propose a combination of a proximal step and the EM algorithm.
	The proximal step is considered in Subsection \ref{sec:prox} and 
	the EM algorithm in Subsection \ref{sec:EM}.
	
	Let $N$ observations $\mathcal X= (x^1,...,x^N) \in \R^{d,N}$      
	with non-negative real-valued weights $\mathcal W = (w_1,...,w_N)$ be given.
	For simplicity, we assume that $\sum_{i=1}^N w_i=N$.
	Then the weighted negative log-likelihood function of the mixture model \eqref{eq:sMM_model}
	is given by
	\begin{align}\label{eq_likelihood}
		\mathcal L(\alpha,\vartheta|\mathcal X)=-\sum_{i=1}^N w_i\log\Big(\sum_{k=1}^K\alpha_k p_{u_k} (x_{u_k}^i|\vartheta_k)\Big).
	\end{align}
	Since we intend to get a sparse mixture model, we propose to minimize instead of $\mathcal L$ the penalized function
	\begin{align}\label{eq_likelihood_plus_0}
		\mathcal L_\lambda(\alpha,\vartheta)
		\coloneqq \mathcal L(\alpha,\vartheta|\mathcal X)
		+ \lambda \|\alpha\|_0 + \iota_{\Delta_K} (\alpha), \quad \lambda > 0
	\end{align}
	with  the zero ,,norm''
	$\|\alpha\|_0 \coloneqq |\{k: \alpha_k >0\}|$ 
	and 
	the indicator function
	$\iota_{\Delta_K}(\alpha) = 0$ if $\alpha \in \Delta_K$ and 
	$\iota_{\Delta_K}(\alpha) = + \infty$ otherwise.
	Here we suppose that the $u_k\subseteq[d]$, $k=1,...,K$ are fixed. 
	In Section \ref{sec:find_u}, we will suggest a heuristic for determining  appropriate sets $u_k$.
	We propose to minimize \eqref{eq_likelihood_plus_0} by alternating between
	the EM steps for $\mathcal L$ and a proximity step for the function
	\begin{equation} \label{prox}
		h(\alpha) \coloneqq \|\alpha\|_0 + \iota_{\Delta_K} (\alpha).
	\end{equation}
	More precisely, we will iterate
	\begin{align} \label{a1}
		(\alpha^{(r+\tfrac12)}, \vartheta^{(r+1)}) &= \mathrm{EM} (\alpha^{(r)}, \vartheta^{(r)}),\\
		\alpha^{(r+1)} &= \prox_{\gamma h}(\alpha^{(r+\tfrac12)}), \quad \gamma > 0, \label{a2}
	\end{align}
	where $\prox_{\gamma h}(\cdot)$ is defined according to~\eqref{def:prox_operator}.
	
	In the following subsection, we consider the  proximity step
	before we explain the EM algorithm for our mixture models with components from Example \ref{ex:distr}.

	%---------------------------------------------------------------------------------------
	\subsection{Proximal Algorithm} \label{sec:prox}
	%---------------------------------------------------------------------------------------
	For a proper, lower semi-continuous function $g\colon \R^K \rightarrow \R\cup \{+\infty\}$ 
	and $\gamma >0$, 
	the \emph{proximal operator}
	$\prox_{\gamma g}$ is defined by 
	\begin{equation}\label{def:prox_operator}
		\prox_{\gamma g} (x) \coloneqq \argmin_{y\in\R^d} \{ \tfrac1{2\gamma} \|x-y\|^2+g(y)\}.
	\end{equation}
	Note, that for a non-convex function $g$ the $\argmin$ is not necessarily single-valued, 
	such that the proximal operator is set-valued. 
	For the non-convex function $g$ in \eqref{prox},
	we can compute a proximum using the following lemma.
	
	\begin{lemma}\label{lem_prox_0}
		Let $\alpha\in\Delta_K$ and assume without loss of generality 
		that $\alpha_1\leq\cdots\leq\alpha_K$. Let $h$ be defined by \eqref{prox}.
		Then the following holds true.
		\begin{enumerate}[i)]
			\item  An element
			$$
			\hat\alpha\in\prox_{\gamma h}(\alpha)
			$$
			is given by $\hat\alpha_J=0$ and 
			$$
			\hat\alpha_{J^c}=\alpha_{J^c}+\frac{1}{|J^c|}\sum_{k\in J} \alpha_k,
			$$
			where $J=[K_0]$ and $J^c$ is defined by $J^c=[K]\backslash J$ and
			\begin{align}
				K_0\in\argmin_{n\in\{0,...,K-1\}}g(n),\quad g(n)=\frac{1}{2\gamma}\frac{\Big(\sum_{k=1}^n \alpha_k\Big)^2}{K-n}+\frac{1}{2\gamma}\sum_{k=1}^n\alpha_k^2-n.\label{eq_to_mini_K_0}
			\end{align}
			\item 
			Assume that $\alpha_i=0$ for some $i\in\{1,...,K\}$. Then it holds for any $\hat \alpha \in\prox_{\gamma h}(\alpha)$ that $\hat \alpha_i=0$.
		\end{enumerate}
	\end{lemma} 
	
	\begin{proof}
		First we note that for $\hat\alpha\in\prox_{\gamma h}(\alpha)$ 
		it holds 
		$$
		\hat\alpha\in\argmin_{y\in\Delta_K}\{\tfrac1{2\gamma}\|\alpha-y\|^2+\|y\|_0\}.
		$$
		With $J\coloneqq\{k\in[K]:\hat\alpha_k=0\}$, this can be rewritten as
		$$
		\hat\alpha_{J^c}\in\argmin_{y\in\Delta_{|J^{c}|}}\{\tfrac1{2\gamma}\|\alpha_{J^c}-y\|^2\},
		$$
		which is the orthogonal projection of $\alpha_{J^c}$ onto $\Delta_{|J^c|}$. Since $\alpha_k\geq0$ for all $k$ and $\|\alpha_{J^c}\|_1\leq1$, this projection is given by 
		$$
		\hat\alpha_{J^c}=\alpha_{J^c}+\frac{1}{|J^c|}\sum_{k\in J} \alpha_k.
		$$
		In particular, we have that 
		$$
		\|\hat\alpha-\alpha\|^2=|J^c|\bigg(\frac{\sum_{k\in J} \alpha_k}{|J^c|}\bigg)^2+\sum_{k\in J}\alpha_k^2=\frac{\Big(\sum_{k\in J} \alpha_k\Big)^2}{|J^c|}+\sum_{k\in J}\alpha_k^2.
		$$
		\begin{enumerate}[(a)]
			\item By the previous calculations, the definition of proximal operators and the fact that $J\subsetneq [K]$ (otherwise $\hat\alpha=0\not\in\Delta_K$), it is left to show that $J$ given by \eqref{eq_to_mini_K_0} is a minimizer of 
			\begin{align} 
				\min_{J\subsetneq[K]}\frac{1}{2\gamma}\frac{\Big(\sum_{k\in J} \alpha_k\Big)^2}{|J^c|}+\frac{1}{2\gamma}\sum_{k\in J}\alpha_k^2+|J^c|.\label{eq_to_mini_J}
			\end{align}
			Due to the monotonicity of \eqref{eq_to_mini_J} in $\alpha_k$, $k\in J$ and $\alpha_1\leq\cdots\leq\alpha_K$, there exists a minimizer of the form $J=[K_0]$ for
			\begin{align} 
				K_0\in\argmin_{n\in\{0,...,K-1\}}g(n),\quad g(n)=\frac{1}{2\gamma}\frac{\Big(\sum_{k=1}^n \alpha_k\Big)^2}{K-n}+\frac{1}{2\gamma}\sum_{k=1}^n\alpha_k^2-n.
			\end{align}
			\item Let $\alpha_k=0$ and assume that $\hat\alpha_k>0$, where $\hat \alpha_k\in\prox_{\gamma(\|\cdot\|_0+\iota_{\Delta_K})}(\alpha)$.
			If there exists no $l\in[K]$ with $\alpha_l>0$ and $\hat\alpha_l=0$, then it holds
			$$
			\frac{1}{2\gamma}\|\hat\alpha-\alpha\|^2+\|\hat\alpha\|_0>\|\alpha\|_0,
			$$
			which is a contradiction to the definition of the proximal operator. 
			Thus, we can assume that such an $l\in[K]$ exists with $\alpha_l > 0$, $\hat \alpha_l = 0$.
			Now, define $\tilde\alpha$ with $\tilde \alpha_k=0$, $\tilde \alpha_l=\hat\alpha_k$ and $\tilde\alpha_j=\hat\alpha_j$ for $j\in[K]\backslash\{k,l\}$. Then it holds by the definition of the proximal operator that
			\begin{align}
				0\geq&\frac{1}{2\gamma}\|\hat\alpha-\alpha\|^2+\|\hat\alpha\|_0-\frac{1}{2\gamma}\|\tilde\alpha-\alpha\|^2-\|\tilde\alpha\|_0\\
				=&\frac{1}{2\gamma}\big((\hat\alpha_k)^2+(\alpha_l)^2-(\tilde\alpha_l-\alpha_l)^2\big)\\
				=&\frac{1}{2\gamma}\big((\hat\alpha_k)^2+(\alpha_l)^2-(\hat\alpha_k-\alpha_l)^2\big).
			\end{align}
			Since $\hat\alpha_k,\alpha_l\in(0,1]$ we have that $|\hat\alpha_k-\alpha_l|<\max(\hat\alpha_k,\alpha_l)$, which implies that the right hand side of the above equation is strictly greater than $0$. This is a contradiction and the proof is completed.
		\end{enumerate}
	\end{proof}
	
	\begin{remark}
		Since sorting the components of $\alpha$ can be done in $\mathcal O(K\log(K))$, 
		the lemma implies, that we can compute an element of 
		$\prox_{\gamma h}(\alpha)$ in $\mathcal O(K \log(K))$. 
		In particular, the computation of the $\prox$ step is very cheap compared with the EM-step.
	\end{remark}
	
	%---------------------------------------------------------------------------------------
	\subsection{EM Algorithm} \label{sec:EM}
	%---------------------------------------------------------------------------------------
	For minimizing the negative log-likelihood function 
	\eqref{eq_likelihood} we will apply an
	EM algorithm, see \cite{Byrne2017,DLR1977} and 
	for a good brief introduction also \cite{Laus2019}.
	We need two different variants of this algorithm, namely 
	for products of von Mises distributions and the wrapped Gaussians.
	
	Let $X^1,...,X^N$ be i.i.d.\ random variables distributed according to $p_X(\cdot|\alpha,\vartheta)$ and $\mathbf{X} = (X^1,\ldots,X^N)$.
	Given a realization $\mathcal X= (x^1,...,x^N) \in  \R^{d,N}$ of $\mathbf{X}$,
	the common idea of the EM algorithm for finding a maximizer of the log-likelihood function
	$$
	\mathcal{L}(\alpha,\vartheta|\mathcal X) 
	= \prod_{i=1}^N p_{X}(x^i|\alpha,\vartheta) 
	$$
	is to introduce an artificial, hidden random variable $Z$
	and to perform the following two steps:
	\\[1ex]
	\textbf{E-Step:} For a fixed estimate $(\alpha^{(r)},\vartheta^{(r)})$ of $(\alpha,\vartheta)$, 
	we approximate the log-likelihood function\\ $\log \left(p_{\mathbf X,Z}(\mathcal X,\mathcal Z|\alpha,\vartheta) \right)$ 
	of the unknown joint realization $(\mathcal X,\mathcal Z)$ by the so-called $Q$-function
	$$
	Q\left( (\alpha,\vartheta),(\alpha^{(r)}, \vartheta^{(r)}) \right) 
	= 
	\E_{(\alpha^{(r)}, \vartheta^{(r)})} \left(\log\left(p_{\mathbf X,Z} (\mathbf X,Z|\alpha, \vartheta) \right)| \mathbf X = \mathcal X \right),
	$$
	where the expectation value is taken with respect to the probability distribution 
	associated with the mixture model $p(\cdot|\alpha^{(r)}, \vartheta^{(r)})$.
	\\
	\textbf{M-Step:} Update $\vartheta$ by maximizing the $Q$-function 
	$$
	\left( \alpha^{(r+1)},\vartheta^{(r+1)} \right) \in\argmax_{\alpha,\vartheta\in\Delta_K \times \Theta}\{Q\left( (\alpha,\vartheta),(\alpha^{(r)},\vartheta^{(r)}) \right)\}.
	$$
	
	A convergence analysis of the EM algorithm via Kullback-Leibler proximal point algorithms 
	was given in \cite{CH2000, CH2008,Laus2019}.
	These convergence results apply also for our special mixture models in the following paragraphs.
	
	\begin{proposition}\label{cor_conv_EM}
		Let the sequence $(\alpha^{(r)},\vartheta^{(r)})_r$ be generated by the above EM steps. 
		Then, the sequence of negative log-likelihood values $\mathcal L(\alpha^{(r)},\vartheta^{(r)}|\mathcal X)$ is decreasing.
	\end{proposition}

	The EM algorithm for minimizing the log-likelihood function of the mixture model
	with products of von Mises functions as summands 
	can be realized with a standard approach \cite{Banerjee2005ClusteringOT, mardia2009directional} for mixture models which uses a special hidden random variable $Z$. 
	In particular, the maximum in the M-step can be computed analytically.
	Unfortunately, with this approach, the M-step has no analytical solution in the
	wrapped Gaussians setting, so that we have to choose the hidden variable $Z$ in a different way.
	We describe both approaches in the following.
	
	%---------------------------------------------------------------------------------------------
	\paragraph{EM Algorithm for products of von Mises distributions}
	%---------------------------------------------------------------------------------------------
	For mixture models, it is common to choose
	hidden variables $Z_{k}^i$ with $Z_{k}^i=1$ if $X^i$ belongs to the $k$-th 
	component of the mixture model and $Z_{k}^i=0$ otherwise. 
	Let $\mathcal X=(x^i)_i$ and $\mathcal Z=(z_{k,l}^i)_{i,k,l}$ be (unknown) joint realizations. 
	\\[1ex]
	\textbf{E-Step:}
	Then, it can be shown, see~\cite{DLR1977,Laus2019} that with the so-called \emph{complete weighted log likelihood function}
	$$
	\ell(\alpha,\vartheta|\mathcal X,\mathcal Z)
	\coloneqq
	\sum_{i=1}^N w_i \sum_{k=1}^K z_k^i \log \left( \alpha_k p_{u_k}(x^i_{u_k}|\vartheta_k) \right)
	$$
	the $Q$ function reads as
	\begin{align}
		Q\left( (\alpha,\vartheta),  (\alpha^{(r)},\vartheta^{(r)}) \right)
		&=
		\E_{(\alpha^{(r)},\vartheta^{(r)})} \left( \ell(\alpha,\vartheta| \mathbf X, Z)|\mathbf X = \mathcal X\right)\\
		&= 
		\sum_{i=1}^N w_i \sum_{k=1}^K \beta^{(r)}_{i,k} \log \left( \alpha_k p_{u_k}(x^i_{u_k}|\vartheta_k) \right)
	\end{align}
	with
	$$
	\beta_{i,k}^{(r)}=
	\frac{\alpha_k^{(r)} p_{u_k} (x^i_{u_k}|\vartheta_k^{(r)})}{\sum_{j=1}^K\alpha_j^{(r)}p_{u_j}(x^i_{u_j}|\vartheta_j^{(r)})}.
	$$
	Note that $\beta_{i,k}^{(r)}$ is an estimate of $z_{k}^i$. Therefore it  can be seen as the probability that $x^i$ 
	arises from the $k$-th summand of the mixture model.
	
	The optimization in the M-step can be done separately for $\alpha$ and $\vartheta$. This results in the 
	EM algorithm for mixture models in Algorithm \ref{alg_em_mm}.
	
	%--------------------------------------------------------------------------------
	\begin{algorithm}[!ht]
		\caption{EM Algorithm for Mixture Models}\label{alg_em_mm}
		\begin{algorithmic}
			\State Input: $(x^1,...,x^N)\in\T^{d, N}$, $(w_1,\ldots,w_N) \in \mathbb R^N$ and 
			initial estimate $\alpha^{(0)},\vartheta^{(0)}$.
			\For {$r=0,1,...$}
			\State \textbf{E-Step:} For $k=1,...,K$ and $i=1,\ldots,N$ compute 
			$$
			\beta_{i,k}^{(r)}
			=\frac{\alpha_k^{(r)}p_{u_k}(x_{u_k}^i|\vartheta_k^{(r)})}{\sum_{j=1}^K\alpha_j^{(r)}p_{u_j}(x_{u_j}^i|\vartheta_j^{(r)})}
			$$
			\State \textbf{M-Step:} For $k=1,...,K$ compute
			\begin{align}
				\alpha_k^{(r+1)}&=\frac1N \sum_{i=1}^N w_i \beta_{i,k}^{(r)},\\
				\vartheta_k^{(r+1)}&=\argmax_{\vartheta_k}\Big\{\sum_{i=1}^{N} w_i \beta_{i,k}^{(r)} \log(p_{u_k} (x^i_{u_k}|\vartheta_k))\Big\}.
			\end{align}
			\EndFor
		\end{algorithmic}
	\end{algorithm}
	%--------------------------------------------------------------------------------
	
	It remains to maximize the function in the second M-step.
	For the von Mises model in Example \ref{ex:distr} iii) this can be done analytically 
	as described in the following.
	\\[1ex]
	\textbf{M-Step:} For products of von Mises distributions, 
	the log-density functions in each component of the mixture model are given by
	$$
	\log \left( p_{u_k} (x_{u_k}) \right) = \sum_{j \in u_k} \log\left(p_M (x_j|\mu_{j,k},\kappa_{j,k}) \right).
	$$
	Then
	\begin{align}
		(\mu_k^{(r+1)},\kappa_k^{(r+1)})
		=
		\argmax_{\mu_k,\kappa_k}\Big\{\sum_{i=1}^N w_i\sum_{j\in u_k}\beta_{i,k}^{(r)}\log(p_M(x^i_j|\mu_{j,k},\kappa_{j,k}))\Big\}
	\end{align}
	decouples for $j$.  For the univariate von Mises distribution the log maximum likelihood
	estimator is well-known  \cite{JS2001} and we obtain
	$$
	\mu_{j,k}^{(r+1)}=\frac{1}{2\pi}\arctan^*\Big(\tfrac{S_{j,k}^{(r)}}{C_{j,k}^{(r)}}\Big),\quad \kappa_{j,k}^{(r+1)}=A^{-1}(R_{j,k}^{(r)}),
	$$
	where $A(\kappa)\coloneqq \tfrac{I_1(\kappa)}{I_0(\kappa)}$,
	\begin{align}
		C_{j,k}^{(r)}&\coloneqq\sum_{i=1}^Nw_i\beta_{i,k}^{(r)}\cos(2\pi x^i_j),\quad 
		S_{j,k}^{(r)}\coloneqq\sum_{i=1}^Nw_i\beta_{i,k}^{(r)}\sin(2\pi x^i_j),\\
		R_{j,k}^{(r)} &\coloneqq\tfrac{1}{N\alpha_k^{(r+1)}}\sqrt{(S_{j,k}^{(r)})^2+(C_{j,k}^{(r)})^2}
	\end{align}
	and
	$\arctan^*\colon \R\times\R\to[0,2\pi)$ with $(S,C)\mapsto \arctan^*\big(\tfrac{S}{C}\big)$ 
	denotes the "quadrant specific" inverse of the tangent defined by
	\begin{align}
		\arctan^*\big(\tfrac{S}{C}\big)=\begin{cases}
			\arctan\big(\tfrac{S}{C}\big),&$if$\quad C>0,\quad S\geq 0,\\
			\tfrac{\pi}{2},&$if$\quad C=0,\quad S>0,\\
			\arctan\big(\frac{S}{C}\big)+\pi,&$if$\quad C<0,\\
			\arctan\big(\frac{S}{C}\big)+2\pi,&$if$\quad C>0,\quad S<0,\\
			\tfrac{3\pi}{2},&$if$\quad C=0,\quad S<0,\\
			\text{undefined},&$if$\quad C=S=0.
		\end{cases}\label{eq_atan2}
	\end{align}
	For the function $A$ it is known that $A$ is a strictly increasing, strictly concave with derivative 
	$A'(\kappa)=(1-\tfrac{A(\kappa)}{\kappa}-A^2(\kappa))> 0$ and has the limits $A(\kappa)\to0$ for $\kappa\to0$ 
	and $A(\kappa)\to1$ for $\kappa\to\infty$, see \cite{JS2001}.
	Thus, we can compute the updates of $\kappa$ using Newtons method.
	
	The whole EM algorithm is summarized in Algorithm \ref{alg_SPAMM_vM} in the appendix.
	
	%----------------------------------------------------------------------------------------------------------------------
	\paragraph{EM Algorithm for Wrapped Gaussians}
	%----------------------------------------------------------------------------------------------------------------------
	For the wrapped Gaussians, the components of the log-likelihood function \eqref{eq_likelihood}
	read as
	$$
	p_{u_k} (x_{u_k}^i|\vartheta_k) = \mathcal N_{w} (x_{u_k}^i|\mu_k,\Sigma_k), \quad \mu_k \in \T^{n}, \, \Sigma_k \in \SPD(n),
	$$
	where $n = |u_k|$.
	Unfortunately, the maximizer in the second M-step of the EM Algorithm \ref{alg_em_mm} cannot be computed analytically
	for the wrapped Gaussian distribution.
	Therefore we adapt the EM by choosing the variable $Z$ in an appropriate way.
	Note, that the resulting EM algorithm is similar to an EM algorithm for non-sparse mixtures of wrapped Gaussians,
	which was already sketched, e.g.\ in \cite{AS2009}.
	\\[1ex]
	\textbf{E-step:} Let $X^1,...,X^N$ be i.i.d.\ random variables.
	We assign to each $X^i$ a label $W_{k}^i$ with $W_{k}^i=1$ if $X^i$ belongs to the $k$-th 
	component of the mixture model and $W_{k}^i=0$ otherwise. 
	Further, recall that for a random variable $Y\sim \mathcal N(\mu,\Sigma)$ 
	it holds that $Y-\lfloor Y \rfloor\sim \wN(\mu,\Sigma)$. 
	Thus, we assign for each $X^i$ a random variable $Y^i$ 
	such that the conditional distribution of $(Y^i|W_{k}^i=1)$ 
	is given by the distribution $\mathcal N(\mu_k,\Sigma_k)$ 
	and it holds $X^i=Y^i-\lfloor Y^i \rfloor$. 
	Now we use as hidden variables in the EM algorithm the random variables $Z_{k,l}^i$, where $Z_{k,l}^i=1$ 
	if $\lfloor Y^i\rfloor=l$ and $W_{k}^i=1$ for $l\in\Z^{|u_k|}$ and 
	set $Z_{k,l}^i =0$ otherwise.
	Let $\mathcal X=(x^i)_i$ and $\mathcal Z=(z_{k,l}^i)_{i,k,l}$.
	Then, with the appropriate complete weighted log likelihood function
	$$
	\ell(\alpha,\mu,\Sigma|\mathcal X,\mathcal Z)
	\coloneqq
	\sum_{i=1}^N w_i \sum_{k=1}^K\sum_{l\in\Z^{|u_k|}}
	z_{k,l}^i \log(\alpha_k\mathcal N(x^i_{u_k}+l|\mu_k,\Sigma_k)).
	$$
	the $Q$-function reads as
	\begin{align}
		&Q((\alpha,\mu,\Sigma),(\alpha^{(r)},\mu^{(r)},\Sigma^{(r)}))\\
		&=\E_{(\alpha^{(r)},\mu^{(r)},\Sigma^{(r)})}(\ell(\alpha,\mu,\Sigma|X,Z)|X=\mathcal X)\\
		&=\sum_{i=1}^Nw_i\sum_{k=1}^K\sum_{l\in\Z^{|u_k|}}
		\E_{(\alpha^{(r)},\mu^{(r)},\Sigma^{(r)})}\big(Z_{k,l}^i|X=\mathcal X\big)\log(\alpha_k\mathcal N(x^i_{u_k}+l|\mu_k,\Sigma_k))\\
		&=\sum_{i=1}^Nw_i\sum_{k=1}^K\sum_{l\in\Z^{|u_k|}}\beta_{i,k,l}^{(r)}\log(\alpha_k\mathcal N(x^i_{u_k}+l|\mu_k,\Sigma_k)),
		\label{q-fun}
	\end{align}
	where by the definition of conditional probabilities 
	\begin{align}
		\beta_{i,k,l}^{(r)}&=\E_{(\alpha^{(r)},\mu^{(r)},\Sigma^{(r)})}(Z_{k,l}^i|X^i=x^i)
		=
		P_{(\alpha^{(r)},\mu^{(r)},\Sigma^{(r)})}(Z_{k,l}^i=1|X^i=x^i)\\
		&=\frac{\alpha_k^{(r)}
			\mathcal N(x^i_{u_k}+l|\mu^{(r)},\Sigma^{(r)})}{\sum_{j=1}^K\alpha_j^{(r)} \mathcal N_w(x^i_{u_j}|\mu^{(r)},\Sigma^{(r)})}.
	\end{align}
	\textbf{M-step:} 
	Analogously as in the EM algorithm for Gaussian Mixture models, the maximizer of the $Q$ function is given by
	\begin{align}
		\alpha^{(r+1)}_k 
		&= \frac1N \sum_{i=1}^Nw_i\sum_{l\in\Z^{|u_k|}} \beta_{i,k,l}^{(r)},\\
		\mu_k^{(r+1)} 
		&=\frac{1}{N\alpha_k^{(r+1)}} \sum_{i=1}^Nw_i\sum_{l\in\Z^{|u_k|}}\beta_{i,k,l}^{(r)}(x_{u_k}^i+l),\\
		\Sigma_k^{(r+1)}
		&=\frac{1}{N\alpha_k^{(r+1)}}\sum_{i=1}^Nw_i\sum_{l\in\Z^{|u_k|}}
		\beta_{i,k,l}^{(r)}(x_{u_k}^i+l-\mu_k^{(r+1)})(x_{u_k}^i+l-\mu_{k}^{(r+1)})^\tT.
	\end{align}
	
	Unfortunately, for every $i=1,...,N$ and $k=1,...,K$, 
	there are infinity many $\beta_{i,k,l}^{(r)}$. 
	But by definition, the $\beta_{i,k,l}^{(r)}$ decay exponentially for $|l|\to\infty$, 
	since by definition the inverse covariance matrix $\Sigma^{-1}$ of the wrapped normal distribution 
	is positive definite which implies that  $-1/2\left(x+l-\mu\right)^T\Sigma^{-1}\left(x+l-\mu\right)<0$ 
	and goes to $-\infty$ as $|l| \rightarrow \infty$. 
	
	Thus, for numerical purposes it suffices to consider 
	$l\in\{-l_{\mathrm{max}},...,l_{\mathrm{max}}\}^{|u_k|}$ for some $l_{\mathrm{max}}\in\N$.
	In other words, we truncate the infinite sum defining the wrapped Gaussian by
	\begin{align}\label{truncate_eval}
		p_{u_k}(x_{u_k}|\mu_k,\Sigma_k) \approx\sum_{l\in\{-l_\mathrm{max},l_\mathrm{max}\}^{|u_k|}}\mathcal N(x_{u_k}+l|\mu_k,\Sigma_k).
	\end{align}
	Nevertheless, the number of coefficients $\beta_{i,k,l}$ depends exponentially 
	on the dimension $|u_k|$ of the wrapped normal distributions.
	Therefore, the parameter estimation can only be performed for small $|u_k|$
	such that  the evaluation does not become the bottleneck of the computation. 
	
	The whole algorithm is given in the Algorithm \ref{alg_SPAMM} in the appendix.
	
	%----------------------------------------------------------------------------------------------------------------------
	\paragraph{EM Algorithm for Diagonal Wrapped Gaussians}
	%----------------------------------------------------------------------------------------------------------------------
	Using in the mixture models only diagonal wrapped normal distributions as in Example \ref{ex:distr} ii),
	we get rid of the exponential dependence of the algorithm on the dimension. 
	This was already observed in \cite{AS2009,SB2005}.
	We have to maximize the log-likelihood function
	\begin{align}
		\mathcal L(\alpha,\mu,\Sigma|\mathcal X)
		&\coloneqq 
		\sum_{i=1}^N w_i
		\log\Big(
		\sum_{k=1}^K \alpha_k \sum_{l\in\Z^{|u_k|}} \prod_{j \in u_k} \mathcal N ( x_j^i + l_j |\mu_{j,k},\sigma_{j,k}^2) 
		\Big).
	\end{align}
	\\[1ex]
	\textbf{E-step:} This step remains basically the same.
	However, we will see that we can sum up over appropriate values of $\beta_{i,k,l}^{(r)}$ to get
	values $\gamma_{i,k,m,j}^{(r)}$. These values
	can finally be computed efficiently in polynomial dependence on the dimension $d$, see \eqref{poly}.
	\\
	\textbf{M-step:}
	We rewrite the $Q$-function in \eqref{q-fun} as
	\begin{align}
		&Q((\alpha,\mu,\Sigma),(\alpha^{(r)},\mu^{(r)},\Sigma^{(r)}))\\
		=&
		\sum_{i=1}^N \sum_{k=1}^K \sum_{l\in\Z^{|u_k|}} w_i \beta_{i,k,l}^{(r)}\log(\alpha_k)
		+
		\sum_{i=1}^N \sum_{k=1}^K \sum_{l\in\Z^{|u_k|}} \sum_{j\in u_k} w_i  \beta_{i,k,l}^{(r)}\log(\mathcal N(x^i_j+l_j|\mu_{j,k},\sigma_{j,k}^2))
		\\
		=&
		\sum_{i=1}^N \sum_{k=1}^K\sum_{l\in\Z^{|u_k|}} w_i \beta_{i,k,l}^{(r)}\log(\alpha_k)
		+
		\sum_{i=1}^N \sum_{k=1}^K \sum_{j\in u_k}\sum_{m\in\Z} w_i \gamma_{i,k,m,j}^{(r)}\log(\mathcal N(x^i_j+m|\mu_{j,k},\sigma_{j,k}^2)),
	\end{align}
	where 
	$\gamma_{i,k,m,j}^{(r)} \coloneqq \sum_{l\in\Z^{|u_k|},l_j=m} \beta_{i,k,l}^{(r)}$.
	\\
	Then maximizing the $Q$-function gives analogously as for Gaussian mixture models
	$$
	\alpha^{(r+1)}_k=\frac1N\sum_{i=1}^Nw_i\sum_{l\in\Z^{|u_k|}}\beta_{i,k,l}^{(r)}=\frac1N\sum_{i=1}^Nw_i\sum_{m\in\Z}\gamma_{i,k,m,j}^{(r)},\quad\text{for any }j\in u_k
	$$
	and
	\begin{align}
		\mu_{j,k}^{(r+1)} 
		&=\frac{1}{N\alpha_k^{(r+1)}}\sum_{i=1}^Nw_i\sum_{m\in\Z}\gamma_{i,k,m,j}^{(r)}(x^i_j+m),\\
		(\sigma_{j,k}^{(r+1)})^2 &=\frac{1}{N\alpha_k^{(r+1)}}\sum_{i=1}^Nw_i\sum_{m\in\Z}\gamma_{i,k,m,j}^{(r)}(x^i_j+m-\mu_{j,k}^{(r+1)})^2.
	\end{align}
	Further, we can rewrite $\gamma^{(r)}_{i,k,m,j}$ as
	\begin{align}
		\gamma^{(r)}_{i,k,m,j}
		&=\frac{\alpha_k^{(r)}\sum_{l\in\Z^{|u_k|},l_j=m}
			\prod_{s\in u_k}\mathcal N(x^i_s+l_s|\mu_{s,k}^{(r)},(\sigma_{s,k}^{(r)})^2)}{\sum_{t=1}^K\alpha^{(r)}_t
			\prod_{s\in u_t} \mathcal N_w (x^i_s|\mu^{(r)}_{s,t},(\sigma^{(r)}_{s,t})^2)}
		\\
		&=
		\frac{\alpha_k^{(r)}\sum_{l\in\Z^{|u_k|},l_j=m}
			\prod_{s\in u_k}\mathcal N(x^i_s+l_s|\mu_{s,k}^{(r)},(\sigma_{s,k}^{(r)})^2)}{\sum_{t=1}^K\alpha^{(r)}_t
			\prod_{s\in u_t}\mathcal N_w(x^i_s|\mu^{(r)}_{s,t},(\sigma^{(r)}_{s,t})^2)}
		\\
		&=
		\frac{\alpha_k^{(r)}\mathcal N(x_j^i+m|\mu_{j,k}^{(r)},(\sigma_{j,k}^{(r)})^2)
			\prod_{s\in u_k\backslash \{j\}}
			\Big(\sum_{l_s\in\Z}\mathcal N(x^i_s+l_s|\mu_{s,k}^{(r)},(\sigma_{s,k}^{(r)})^2)\Big)}{\sum_{t=1}^K\alpha^{(r)}_t
			\prod_{s\in u_t} \mathcal N_w(x^i_s|\mu^{(r)}_{s,t},(\sigma^{(r)}_{s,t})^2)}\\
		&=
		\frac{\alpha_k^{(r)}\mathcal N(x_j^i+m|\mu_{j,k}^{(r)},(\sigma_{j,k}^{(r)})^2)
			\prod_{s\in u_k\backslash \{j\}}\mathcal N_w(x^i_s|\mu_{s,k}^{(r)},(\sigma_{s,k}^{(r)})^2)}{\sum_{t=1}^K\alpha^{(r)}_t
			\prod_{s\in u_t} \mathcal N_w(x^i_s|\mu^{(r)}_{s,t},(\sigma^{(r)}_{s,t})^2)}. \label{poly}
	\end{align}
	
	As in the previous algorithm, for any $i$, $k$ and $j$ there are infinity many $\gamma^{(r)}_{i,k,m,j}$. 
	However, with the same justifications as above the $\gamma^{(r)}_{i,k,m,j}$ decay exponentially for $|m|\to\infty$ 
	such that it suffices to consider $m\in\{-m_\mathrm{max},m_\mathrm{max}\}$. 
	While this approximation led to an exponential dependence of the complexity of the previous EM on $\max_{k=1,...,K}|u_k|$, 
	the complexity of the EM algorithm depends only polynomial on $\max_{k=1,...,K}|u_k|$.
	
	The whole algorithm is given in Algorithm \ref{alg_SPAMM_diag} in the appendix.
	
	\paragraph{Convergence Consideration}
	Finally, we return to the coupled proximum-EM algorithm in \eqref{a1} and \eqref{a2}.
	In the following theorem, we restrict our attention to the mixture model 
	with wrapped Gaussians as components, but the statements
	apply for the other two models in Example \ref{ex:distr}, too.
	
	\begin{theorem}
		Let $(\alpha^{(r)},\mu^{(r)},\Sigma^{(r)})_r$ be generated by \eqref{a1}-\eqref{a2} with one EM step as in Algorithm \ref{alg_SPAMM}. 
		Then the following holds true.
		\begin{enumerate}[i)]
			\item Assume that $\alpha_k^{(r_0)}=0$. Then, we have that $\alpha_k^{(r)}=0$ for any $r\geq r_0$. In particular, 
			the number $\|\alpha^{(r)}\|_0$ of non-zero elements in $\alpha$ is monotone decreasing.
			\item There exists some $\tilde \lambda>0$ such that for all $\lambda>\tilde \lambda$ 
			the sequence $\left(\mathcal L_\lambda(\alpha^{(r)},\mu^{(r)},\Sigma^{(r)}) \right)_r$ is monotone decreasing.
			\item Assume, that $\gamma<K-1$. Then, there exists $r_0\in\N$ such that for any $r\geq r_0$ it holds for any $k=1,...,K$ that either $\alpha_k^{(r)}=0$ or $\alpha_k^{(r)}\geq \sqrt{\frac{2\gamma(K_0-1)}{K_0}}$, where $K_0=\min_{r\in\N}\|\alpha^{(r)}\|_0$.
		\end{enumerate}
	\end{theorem}
	
	\begin{proof}
		\begin{enumerate}[i)]
			\item Assume that $\alpha_k^{(r)}=0$. 
			This implies that in Algorithm \ref{alg_SPAMM} it holds $\beta_{i,k,l}^{(r)}=0$ for all $i$ and $l$ 
			and consequently it holds $\alpha^{(r+1/2)}_k=0$. 
			By part ii) of Lemma \ref{lem_prox_0} we obtain that also $\alpha^{(r+1)}_k=0$. 
			Now, induction gives part i).
			
			\item By the first part of the proof, we conclude that the set 
			$R\coloneqq\{r\in\N:\|\alpha^{(r+1)}\|_0<\|\alpha^{(r)}\|_0\}$  
			is finite. 
			Further, by the definition of an $\mathrm{EM}$ step it holds for any $r\in\N$ 
			that $\|\alpha^{(r+1/2)}\|_0=\|\alpha^{(r)}\|_0$. 
			In combination with Propostion \ref{cor_conv_EM} this yields for any $\lambda>0$ that
			\begin{align}\label{eq_descent1}
				\mathcal L_\lambda(\alpha^{(r+1/2)},\mu^{(r+1)},\Sigma^{(r+1)})
				\leq 
				\mathcal L_\lambda(\alpha^{(r)},\mu^{(r)},\Sigma^{(r)}).
			\end{align}
			Thus, we have for $r\not\in R$ that $\|\alpha^{(r+1/2)}\|_0=\|\alpha^{(r+1)}\|_0$, 
			which yields by the definitions the proximal operator and of $h$ in \eqref{prox}
			that $\alpha^{(r+1/2)}=\alpha^{(r+1)}$. 
			Together with \eqref{eq_descent1} we get for any $r\not\in R$ and $\lambda>0$ that
			\begin{align}
				\mathcal L_\lambda(\alpha^{(r+1)},\mu^{(r+1)},\Sigma^{(r+1)})\leq \mathcal L_\lambda(\alpha^{(r)},\mu^{(r)},\Sigma^{(r)}).\label{eq_descent2}
			\end{align}
			For $r\in R$, we have $\|\alpha^{(r+1)}\|_0<\|\alpha^{(r+1/2)}\|_0$. 
			Consequently, we obtain
			\begin{align}
				&\mathcal L_\lambda(\alpha^{(r+1/2)},\mu^{(r+1)},\Sigma^{(r+1)})-\mathcal L_\lambda(\alpha^{(r+1)},\mu^{(r+1)},\Sigma^{(r+1)})\\
				=&\mathcal L(\alpha^{(r+1/2)},\mu^{(r+1)},\Sigma^{(r+1)}|\mathcal X)-\mathcal L(\alpha^{(r+1)},\mu^{(r+1)},\Sigma^{(r+1)}|\mathcal X)\\
				&+\lambda(\|\alpha^{(r+1/2)}\|_0-\|\alpha^{(r+1)}\|_0).
			\end{align}
			This is greater or equal to zero, if
			$$
			\lambda\geq\lambda_r\coloneqq\frac{\mathcal L(\alpha^{(r+1)},\mu^{(r+1)},\Sigma^{(r+1)}|\mathcal X)-\mathcal L(\alpha^{(r+1/2)},\mu^{(r+1)},\Sigma^{(r+1)}|\mathcal X)}{\|\alpha^{(r+1/2)}\|_0-\|\alpha^{(r+1)}\|_0}.
			$$
			Together with \eqref{eq_descent1} we get  for $r\not\in R$ and $\lambda\geq\lambda_r$ that
			\begin{align}
				\mathcal L_\lambda(\alpha^{(r+1)},\mu^{(r+1)},\Sigma^{(r+1)})\leq \mathcal L_\lambda(\alpha^{(r)},\mu^{(r)},\Sigma^{(r)}).\label{eq_descent3}
			\end{align}
			Finally, we set $\tilde\lambda\coloneqq\max_{r\in R}\lambda_r$, which is finite, 
			since $R$ is finite. 
			Combined with \eqref{eq_descent2} and \eqref{eq_descent3} this yields the claim.
			
			\item As in the previous part of the proof, the set 
			$R\coloneqq\{r\in\N:\|\alpha^{(r+1)}\|_0<\|\alpha^{(r)}\|_0\}$ 
			is finite and for $r\not\in R$ holds $\alpha^{(r+1)}=\alpha^{(r+1/2)}$. 
			Now let $r_0 \coloneqq \max_{r\in R} r+1$ and assume that there exists $r\geq r_0$ with 
			$\alpha^{(r)}_k\in\Big(0,\sqrt{\frac{2\gamma(K_0-1)}{K_0}}\Big)$. Then, it holds that $\|\alpha^{(r)}\|_0=K_0$ and
			\begin{align}
				\alpha^{(r)}=\prox_{\gamma h}(\alpha^{(r-1/2)})
				=
				\prox_{\gamma h}(\alpha^{(r)}).\label{eq_prox_equal}
			\end{align}
			Now, define $\tilde \alpha\in\Delta_K$ by $\tilde \alpha_k=0$ and 
			$$
			\tilde \alpha_j=\begin{cases}0,&$if $\alpha_j^{(r)}=0,\\\alpha^{(r)}_j+\frac{\alpha^{(r)}_k}{K_0-1},&$else.$\end{cases}
			$$
			This yields 
			\begin{align}
				\frac{1}{2\gamma}\|\tilde \alpha-\alpha^{(r)}\|^2+\|\tilde\alpha\|_0
				&=\frac{1}{2\gamma}\Big((\alpha^{(r)}_k)^2+(K_0-1)\Big(\frac{\alpha^{(r)}_k}{K_0-1}\Big)^2\Big)+\|\alpha^{(r)}\|_0-1\\
				&=\frac{1}{2\gamma}\frac{K_0}{K_0-1}(\alpha^{(r)}_k)^2+\|\alpha^{(r)}\|_0-1<\|\alpha^{(r)}\|_0.
			\end{align}
			This is a contradiction to \eqref{eq_prox_equal} such that the proof is completed.
		\end{enumerate}
	\end{proof}

	%---------------------------------------------------------------------------------------
	\subsection{Model Selection} \label{sec:find_u}
	%---------------------------------------------------------------------------------------
	%
	The model and optimization algorithm in the previous section assumed, that the $u_k$, $k=1,...,K$ are known. 
	In the following, we propose a heuristic for selecting the $u_k$.
	
	The underlying assumption is, that the distribution of the samples $(x^i)_i$ can be represented 
	by a sparse mixture model
	$$
	p(x)=\sum_{k=1}^Kp(x_{u_k}|\vartheta_k),
	$$
	with a small number of variable interactions, i.e., $|u_k|\leq d_s$, $k=1,...,K$ for some small $d_s\in\{1,...,d\}$.
	Here, we assume that the number $d_s$ is known a priori.
	Fruther, we assume that the number $K$ of required components in the sparse mixture model is small.
	
	For our heuristic, we start with $K=1$ and $u_1=\emptyset$. 
	Then we extend our model iteratively by repeating the following steps $d_s$ times.
	
	\begin{enumerate}
		\item For every $k=1,...,K$, we first compute the probability that the sample $x^i$ belongs to component $k$ of the sparse mixture model. This probability is given by
		$$
		\beta_{i,k}=\frac{\alpha_kp(x_{u_k}^i|\vartheta_k)}{\sum_{j=1}^K \alpha_j p(x_{u_j}^i|\vartheta_j)}.
		$$
		For $m \in[d]\backslash u_k$, we want to test, if we can fit the weighted samples 
		$(x^i)_i$ with importance weights 
		$(w_i\beta_{i,k})_i$ 
		better with a density function 
		$p(x_{u_k\cup \{m\}}^i|\tilde \mu_k,\tilde \Sigma_k)$ as with the density function $p(x_{u_k}|\mu_k,\Sigma_k)$. 
		If the distribution $p(x_{u_k}|\mu_k,\Sigma_k)$ fits the samples perfectly, 
		we have that for any $m \in[d]\backslash u_k$ the samples $(x_m^i)_i$ 
		with importance weights $(w_i\beta_{i,k})_i$ are uniformly distributed and independent from $(x_{u_k}^i)_i$ 
		with weights $(w_i\beta_{i,k})_i$.
		Consequently, we apply two tests.
		First, we apply a Kolmogorov-Smirnov test described in Appendix \ref{sec:test}
		for the hypothesis 
		$$
		H_0\colon \text{The samples } (x_m^i)_i \text{ with weights } (w_i\beta_{ik})_i \text{ are uniformly distributed}
		$$
		against the alternative 
		$$
		H_1\colon \text{The samples } (x_m^i)_i \text{ with weights } (w_i\beta_{ik})_i  \text{ are not uniformly distributed.}
		$$
		The hypothesis $H_0$ is accepted, 
		if the Kolmogorov-Smirnov test statistic, see \eqref{eq_KS_weighted}, fulfills
		$\mathrm{KS}((w_i\beta_{i,k})_i,(x_m^i)_i)<c_1$ 
		for some a priori fixed $c_1\in\R_{>0}$.
		Since it is difficult, to test for independence, our second test is based on the correlation. Here, we test the hypothesis
		$$
		\tilde H_0\colon (w_i\beta_{i,k},x_m^i)_i \text{ and } (w_i\beta_{i,k},x_{u_k}^i)_i \text{ are uncorrelated}
		$$
		against the alternative 
		$$
		\tilde H_1\colon (w_i\beta_{i,k},x_m^i)_i \text{ and } (w_i\beta_{i,k},x_{u_k}^i)_i \text{ are correlated.}
		$$
		The hypothesis $\tilde H_0$ is accepted, if the correlation coefficient
		$$|\mathrm{Corr}((w_i\beta_{i,k},x_m^i)_i,(w_i\beta_{i,k},x_j^i)_i)|<c_2$$
		for all $j\in u_k$ 
		and some a priori fixed $c_2\in\R_{>0}$.
		Now, we set 
		$$
		U_k\coloneqq\{u_k\}\cup\{u_k\cup \{m\}:m\in[d]\backslash u_k \text{ with } 
		H_0\text{ is rejected or } \tilde H_0 \text{ is rejected}\}.
		$$
		and define a new sparse mixture model with $\tilde K=\sum_{k=1}^K|U_k|$ components, 
		where the new $u_i$ are given by the elements of the $U_k$, $k=1,...,K$. 
		For wrapped normal distributions, we initialize the parameters of $u_k\cup \{m\}$ by the following procedure:
		First, we estimate the parameters $(\hat\mu,\hat\sigma^2)$ of a univariate wrapped normal distribution based on the samples $(x_m^i)_i$ with importance weights $(w_i\beta_i,k)_i$.
		Then, we initialize the component with indices $u_k\cup\{m\}$ by the parameters of the distribution of a random variable $X$ characterized by
		$$
		(X_{u_k},X_m)\sim \mathcal N_w(\mu_k,\Sigma_k)\times\mathcal N_w(\hat \mu,\hat\sigma^2),
		$$
		where $(\mu_k,\Sigma_k)$ are the old parameters corresponding to $u_k$.
		%\todo{Initialisierung aufschreiben.}
		\item As a second step, we estimate the parameters $(\alpha,\vartheta)$ of the new sparse mixture model 
		using the iteration \eqref{a1} and \eqref{a2}.
		\item Finally, we reduce the number of components of the sparse mixture model by 
		\begin{enumerate}[i)]
			\item removing all components $k$ with weight $\alpha_k=0$,
			\item replacing the components $k$ and $l$ with $u_k=u_l$ 
			by one component $u_k$ with weight $\alpha_k+\alpha_l$, $\mu_k$ and $\Sigma_k$ 
			if the corresponding distributions are similar.
			As a similarity measure, we use here the \emph{Kullback-Leibler divergence}, which can be approximated by the Monte Carlo method as
			$$
			\mathrm{KL}(p,q)\approx \frac{1}{N_\mathrm{MC}}\sum_{i=1}^{N_{\mathrm{MC}}} \log(p(s_i))-\log(q(s_i)),
			$$
			where the $s_i$ are sampled from the probability distribution corresponding to the density $p$.
		\end{enumerate}
	\end{enumerate}

	%----------------------------------------------------------
	\section{Numerical Results}\label{sec:numerics}
	%----------------------------------------------------------
	In this section, we demonstrate the performance of our algorithm by four numerical examples.
	The implementation is done in Tensorflow and Python. The code is available online\footnote{\url{https://github.com/johertrich/Sparse_Mixture_Models}}.
	
	In the first two subsections, the non-negative density function $f\colon[0,1]^d\to\R$ with $\|f\|_{L^1}=1$ is given as ground truth
	and we can sample from the corresponding distribution.
	More precisely, we consider the following functions:
	\begin{itemize}
		\item[1.] two mixture models,
		\item[2.] the sum of the tensor products of splines, which was also considered in \cite{BPS2020},
		\item[3.] the normalized Friedman-1 function, which was also examined in \cite{BGG2009,BDL2011,MLH2003}.
	\end{itemize}
	The samples in the third subsection are created in a special way and the underlying density function is unknown.
	
	Since the reconstruction quality possibly depends on the random choice of the samples, we repeat this procedure $10$ times. 
	In the first example, we can directly sample from the distribution, while we use rejection sampling for the two other ones, 
	see e.g.\ \cite{ADDJ2003, RC2013}.
	This works out as follows.\\
	Let $M\geq\sup_{x\in[0,1]^d}f(x)$. Now, we generate a candidate by drawing $x$ from the uniform distribution on 
	$\mathcal U_{[0,1]^d}$ and $z$ from $\mathcal U_{[0,1]}$. 
	Then, we accept $x$ as a sample if $z<f(x)/M$ and reject $x$ otherwise.
	It can be shown, that samples $x^i$, $i=1,...,N$ generated by this procedure correspond to the distribution 
	given by the density $f$, see e.g\ \cite[pp. 49]{RC2013}.
	\\
	
	To evaluate our reconstruction results, we compare 
	\begin{itemize}
		\item[-] the value of the log likelihood function
		of the original function $f$ with those of the estimated one $\hat p$, and
		\item[-]  the relative $L_q$, $q =1,2$ errors 
		$$
		e_{L_q}(\hat p,f)=\frac{\|f-\hat p\|_{L_q}}{\|f\|_{L_q}}.
		$$
		Since we cannot calculate the high-dimensional integral $e_{L_q}$ directly, 
		we approximate it via Monte-Carlo integration. 
		That is, we draw $N_{\mathrm{MC}}=100000$ samples $s_1,...,s_{N_{\mathrm{MC}}}$ 
		from the uniform distribution in $[0,1]^d$ and approximate the $L_q$-norms by
		$$
		\|f\|_{L_q}^q\approx\frac{1}{N_{\mathrm{MC}}}\sum_{i=1}^{N_{\mathrm{MC}}}f(s_i)^q,\quad\|f-\hat p\|_{L_q}^q
		\approx
		\frac{1}{N_{\mathrm{MC}}}\sum_{i=1}^{N_{\mathrm{MC}}}(f(s_i)-\hat p(s_i))^q.
		$$
	\end{itemize}

	%--------------------------------------------------------
	\subsection{Samples from Mixture Models}
	%--------------------------------------------------------
	For $d=9$, we consider the ground truth density function
	\begin{equation} \label{num_1}
		f(x) 
		\coloneqq
		\sum_{k=1}^6  \alpha_k p_{u_k} (x_{u_k}|\mu_k,\Sigma_k), 
		\quad
		p_{u_k} (x_{u_k}|\mu_k,\Sigma_k) \coloneqq \sum_{l\in\Z^{|u_k|}}\mathcal N(x_{u_k}+l|\mu_k,\Sigma_k)
	\end{equation}
	with
	\begin{align*}
		(u_1,\ldots,u_6) 
		&\coloneqq (\{0,1\},\{2,3\},\{4,5,6\},\{6,7\},\{8,9\},\{2\}),\\ 
		\alpha
		&\coloneqq(0.2,0.2,0.2,0.2,0.1,0.1),\\
		\mu 
		&\coloneqq \frac12 \left( (1,1)^\tT,(1,1)^\tT,(1,1,1)^\tT,(1,1)^\tT,(1,1)^\tT, 1 \right)
	\end{align*}
	and the following settings of covariance matrices:
	\begin{enumerate}[a)]
		\item Diagonal matrices: for $k = 1,\ldots,6$,
		$$                       
		\Sigma_k = \sigma^2 I_{|u_k|},\quad \sigma^2=0.01.
		$$
		\item Non-diagonal matrices: for $k = 1,2,4,5$,
		$$\Sigma_k \coloneqq \sigma^2\Big(\begin{array}{cc}1&c_k\\c_k&1\end{array}\Big), \quad c_1=c_2=0.5,\, c_4=-0.6,\, c_5=0.1,$$ 
		and
		$$
		\Sigma_3 \coloneqq \sigma^2\Bigg(\begin{array}{ccc}1&0.3&0.2\\0.3&1&0.1\\0.2&0.1&1\end{array}\Bigg),
		\quad 
		\Sigma_6=\sigma^2, 
		\quad
		\sigma^2=0.01.
		$$
	\end{enumerate}
	We do all computations for $N=10000$ and $N=50000$ samples. We iterate the heuristic from Section~\ref{sec:find_u} for $d_s=3$ times.
	Then the average relative $L_q$,  $q=1,2$ errors 
	as well as the log likelihood values are given for the three settings from Example \ref{ex:distr}
	in Table \ref{tab_results_synthetic}.
	Figure~\ref{fig_components} shows a diagram with the weights of the recovered couplings $u_k$.
	More precisely, the value of the bar with label $u$ is given by the sum of all 
	$\alpha_k$, where $u_k=u$ in the reconstruction.
	
	\begin{table}
		\begin{center}
			\resizebox*{!}{2.7cm}{
				\begin{tabular}{c|c|cccc}
					Truth&Method&$\mathcal L_f(x^1,...,x^N)$&$\mathcal L_{\hat p}(x^1,...,x^N)$&$e_{L^1}(\hat p,f)$&$e_{L^2}(\hat p,f)$\\\hline
					a)&wrapped&$7185.2\pm 119.3$&$7244.8\pm 126.7$&$0.0727\pm 0.0062$&$0.0879\pm0.0125$\\
					a)&comp.\ wrapped&$7185.2\pm 119.3$&$7227.0\pm 121.0$&$0.0614\pm 0.0048$&$0.0728\pm0.0064$\\
					a)&von Mises&$7185.2\pm119.3$&$7215.1\pm115.9$&$0.0706\pm0.0036$&$0.0793\pm0.0062$\\\hline
					b)&wrapped&$7825.5\pm\n97.6$&$7879.0\pm\n94.3$&$0.0675\pm0.0083$&$0.0824\pm0.0136$\\
					b)&comp.\ wrapped&$7825.5\pm\n97.6$&$7764.0\pm\n96.2$&$0.1165\pm0.0021$&$0.1128\pm0.0043$\\
					b)&von Mises&$7825.5\pm\n97.6$&$7754.1\pm\n94.4$&$0.1182\pm0.0028$&$0.1135\pm0.0035$
			\end{tabular}}\\
			\vspace{0.3cm}
			\resizebox*{!}{2.7cm}{
				\begin{tabular}{c|c|cccc}
					Truth&Method&$\mathcal L_f(x^1,...,x^N)$&$\mathcal L_{\hat p}(x^1,...,x^N)$&$e_{L^1}(\hat p,f)$&$e_{L^2}(\hat p,f)$\\\hline
					a)&wrapped&$35956\pm 167$&$35852\pm 254$&$0.0484\pm 0.0154$&$0.0701\pm0.0313$\\
					a)&comp.\ wrapped&$35956\pm 167$&$35864\pm 170$&$0.0446\pm 0.0161$&$0.0684\pm0.0317$\\
					a)&von Mises&$35956\pm167$&$35945\pm 179$&$0.0387\pm0.0024$&$0.0390\pm0.0035$\\\hline
					b)&wrapped&$39206\pm272$&$39088\pm295$&$0.0507\pm0.0109$&$0.0781\pm0.0191$\\
					b)&comp.\ wrapped&$39206\pm272$&$38719\pm285$&$0.0971\pm0.0032$&$0.0912\pm0.0046$\\
					b)&von Mises&$39206\pm272$&$38722\pm271$&$0.0966\pm0.0072$&$0.0897\pm0.0064$
				\end{tabular}
			}
		\end{center}
		\caption{Approximation of $f$ in \eqref{num_1} with a) and b) by the three sparse mixture models in Example~\ref{ex:distr}. 
			Average value of the log likelihood function and relative $L_q$, $q=1,2$.
			Top: $N=10000$, bottom: $N=50000$.}
		\label{tab_results_synthetic}
	\end{table}
	
	\begin{figure}
		\begin{subfigure}[t]{0.33\linewidth}
			\centering
			\includegraphics[width=\textwidth]{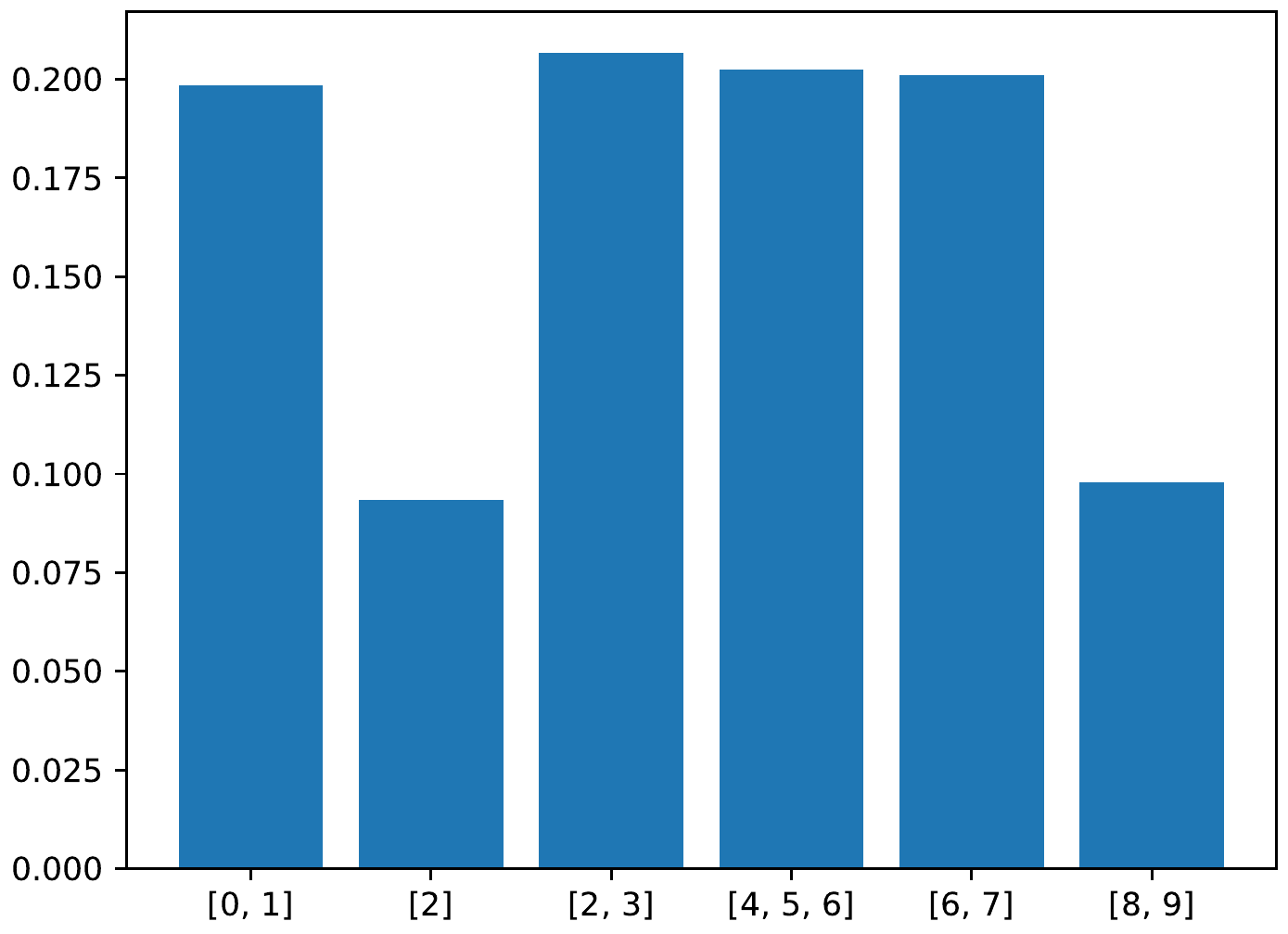}
		\end{subfigure}\hfill
		\begin{subfigure}[t]{0.33\linewidth}
			\centering
			\includegraphics[width=\textwidth]{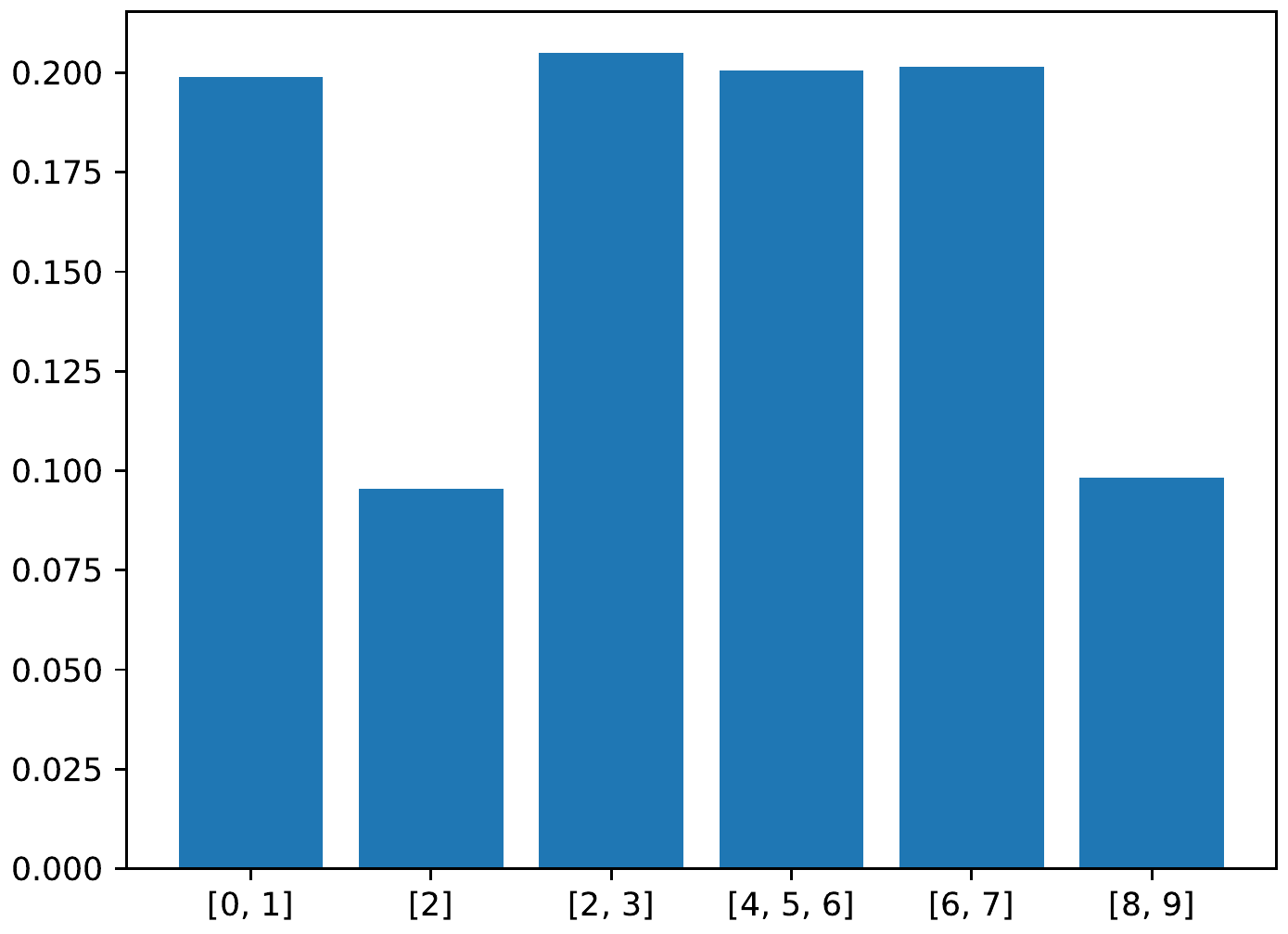}
		\end{subfigure}\hfill
		\begin{subfigure}[t]{0.33\linewidth}
			\centering
			\includegraphics[width=\textwidth]{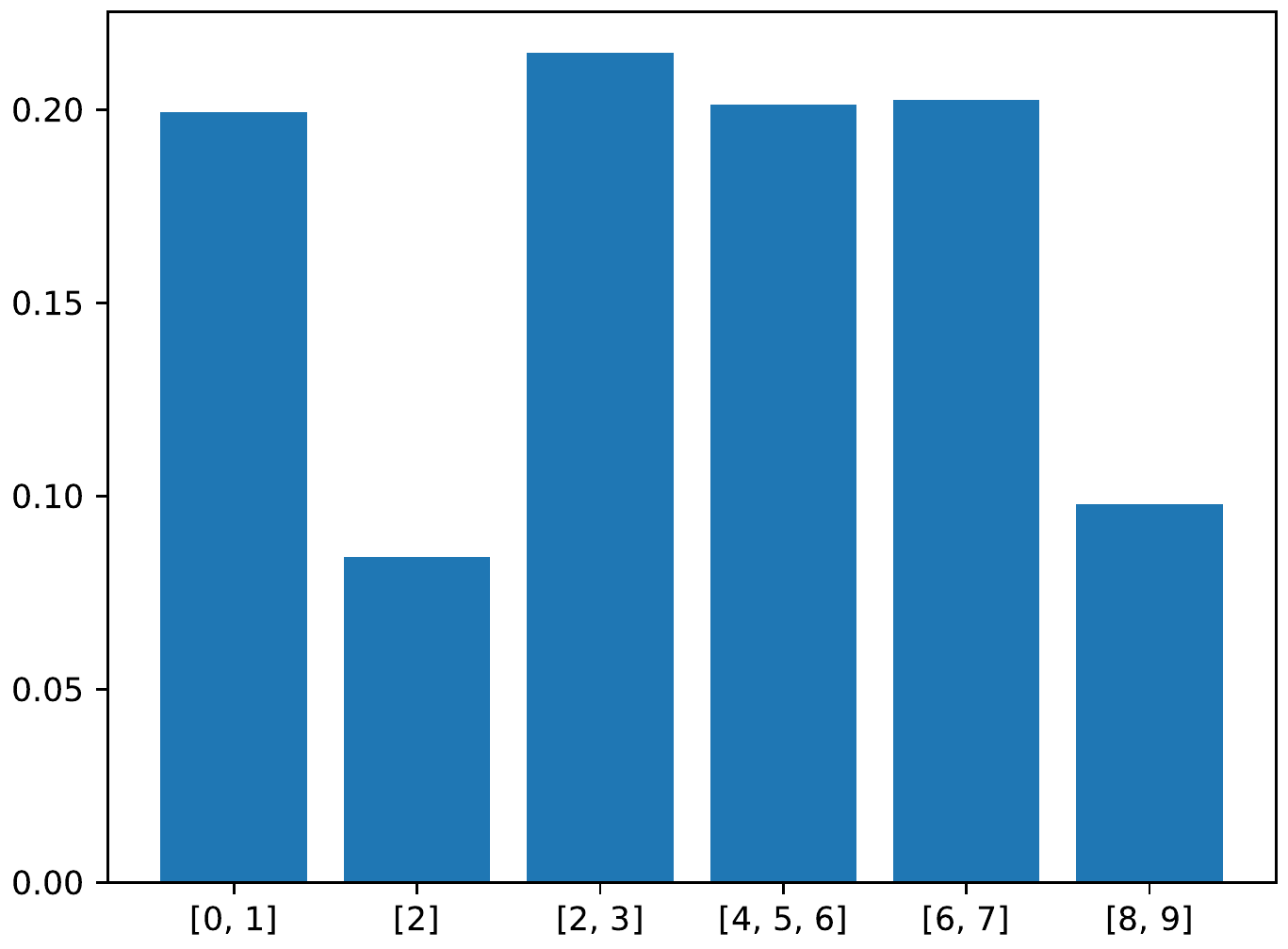}
		\end{subfigure}\hfill
		\begin{subfigure}[t]{0.33\linewidth}
			\centering
			\includegraphics[width=\textwidth]{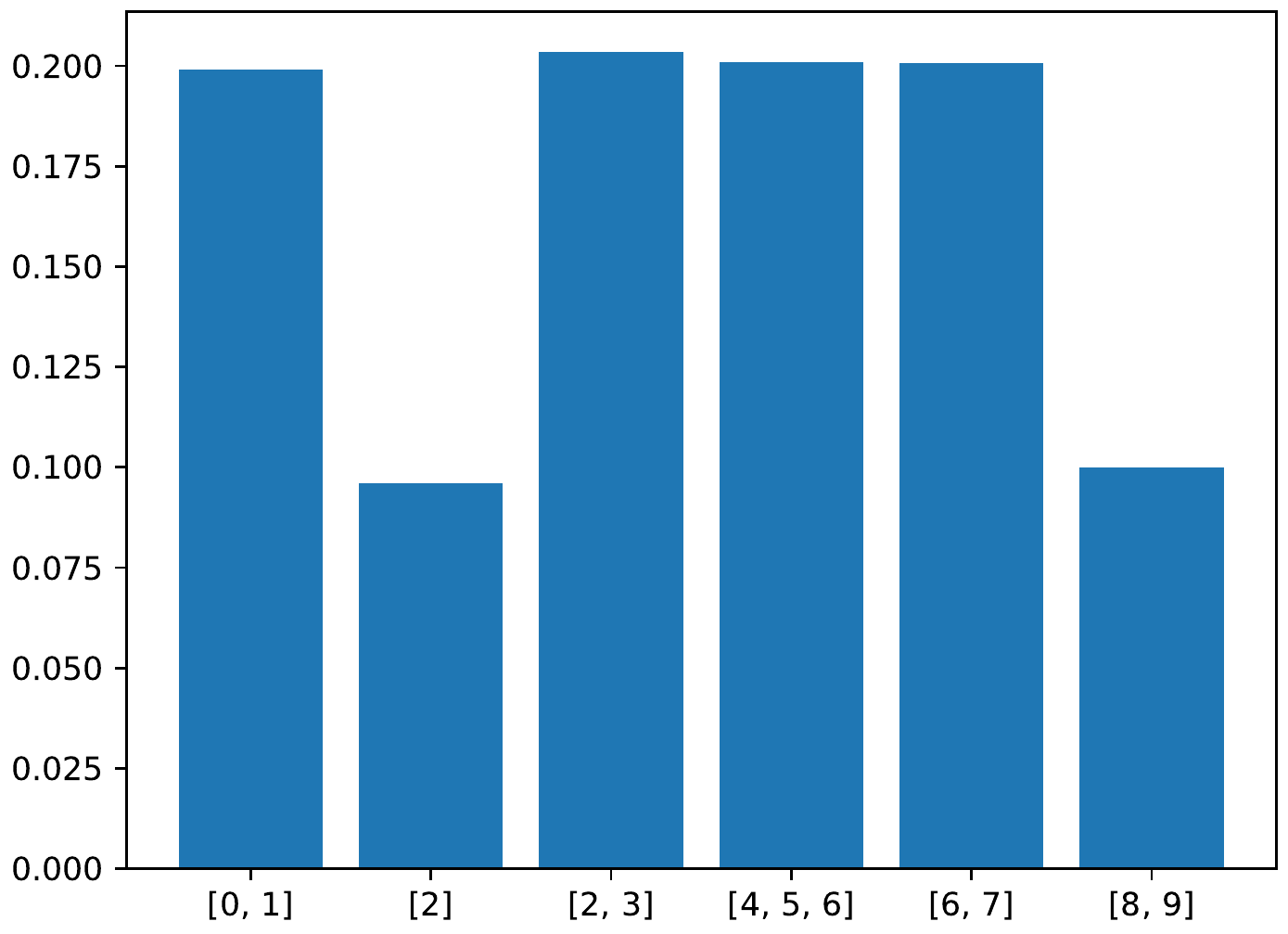}
			\caption*{wrapped}
		\end{subfigure}\hfill
		\begin{subfigure}[t]{0.33\linewidth}
			\centering
			\includegraphics[width=\textwidth]{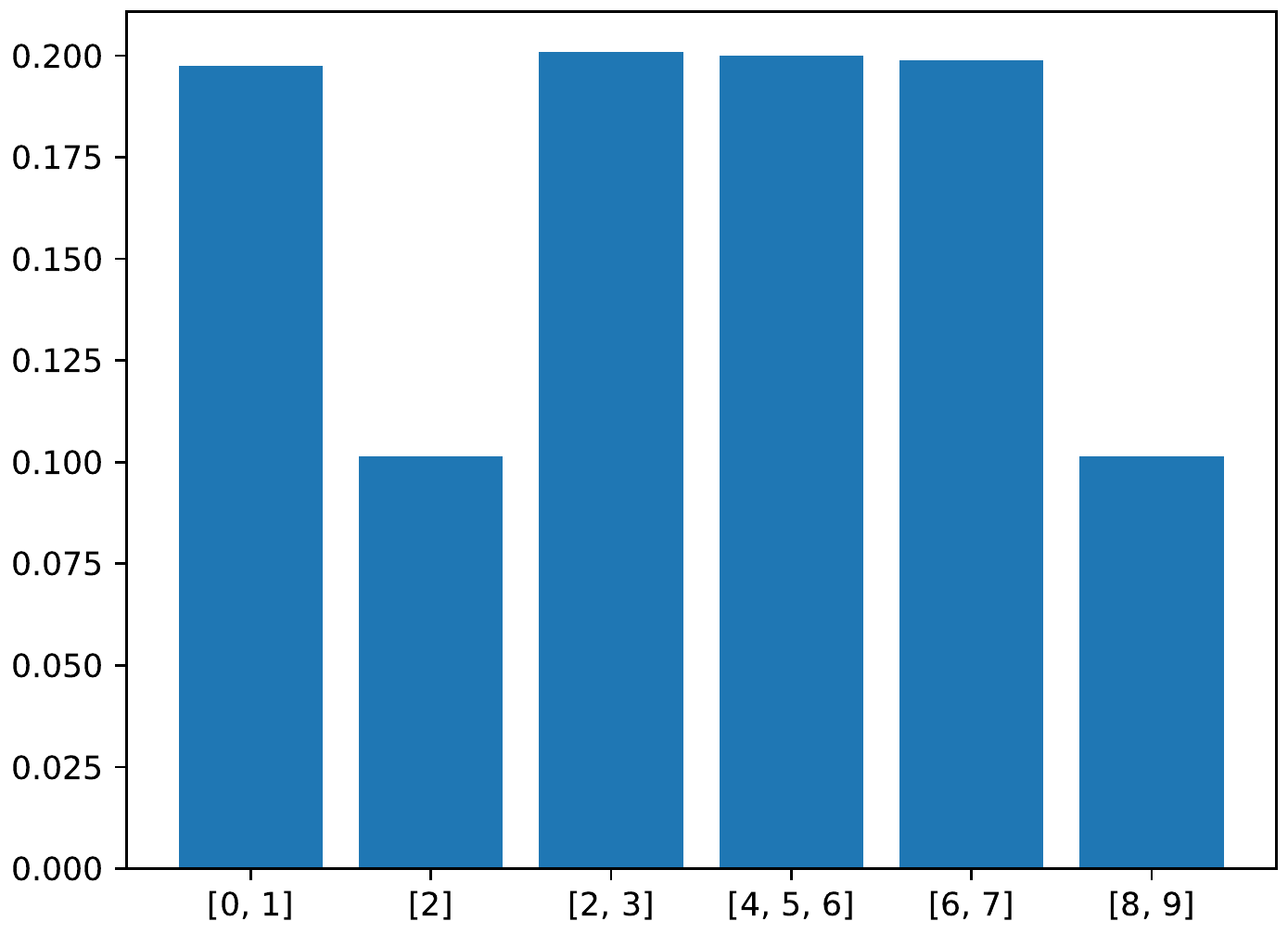}
			\caption*{comp.\ wrapped}
		\end{subfigure}\hfill
		\begin{subfigure}[t]{0.33\linewidth}
			\centering
			\includegraphics[width=\textwidth]{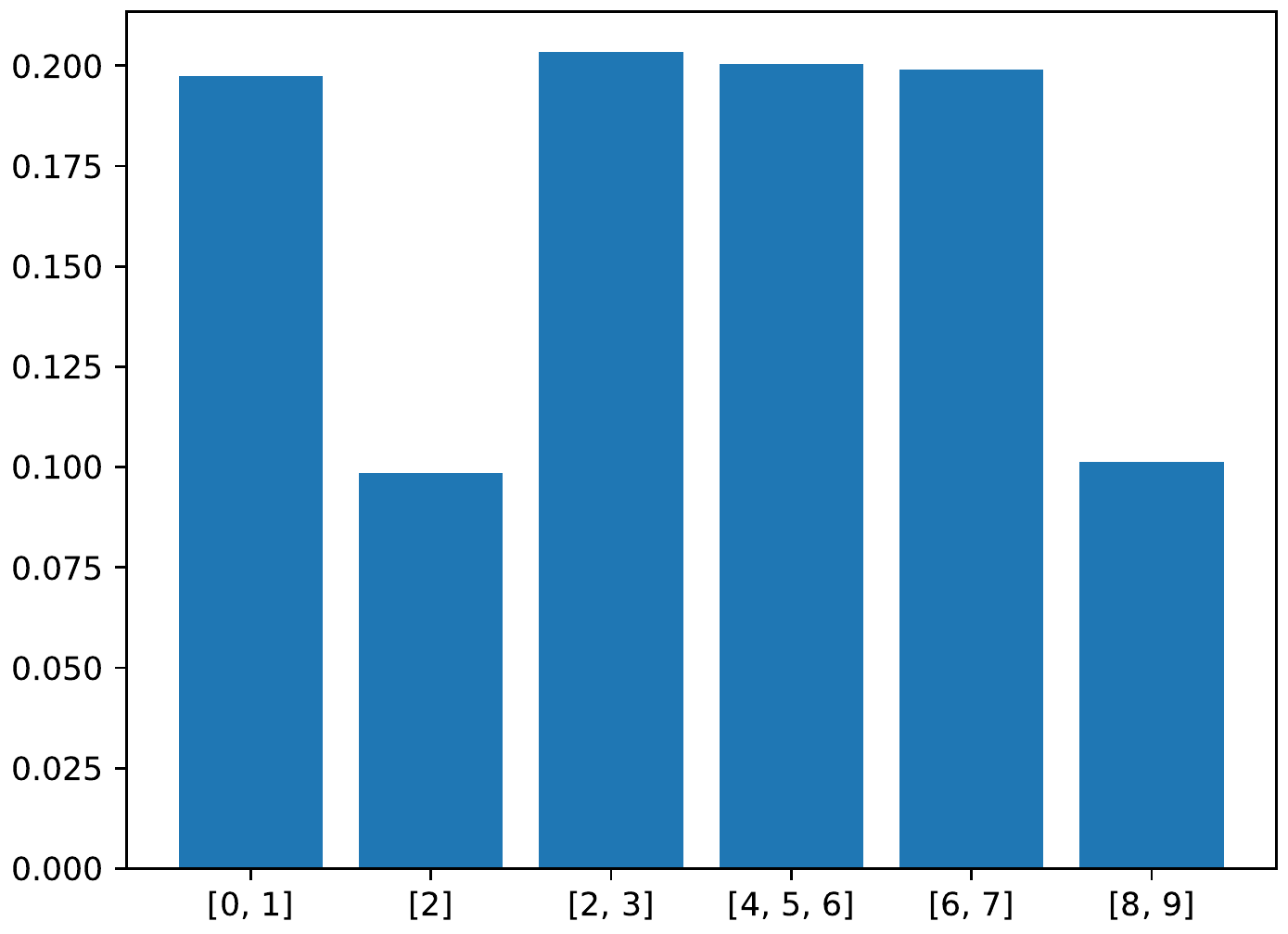}
			\caption*{von Mises}
		\end{subfigure}\hfill
		\caption{Weights of the recovered couplings $u_k$ for the approximation of $f$ in \eqref{num_1} with a) and b) by the three mixture models in Example~\ref{ex:distr} for $N=10000$ samples. Top: model a), bottom: model b).}
		\label{fig_components}
	\end{figure}
	
	We observe that the couplings $u_k$ are reconstructed exactly. 
	Further, the log likelihood value is for the sparse mixture model with full wrapped Gaussian covariance 
	matrices 
	in all cases larger than for the ground truth function.
	Thus, the approximation error is due to the estimation error of the maximum likelihood estimator 
	and therefore due to the lack of information contained in the samples and not due to the approximation method. 
	Further, the sparse mixture model with full wrapped Gaussian covariances admits 
	in all cases a larger log likelihood value than those with the diagonal wrapped Gaussians.
	This should not be surprising, since the sparse mixture model with the wrapped Gaussian covariance matrices contains 
	the other setting.
	
	%---------------------------------------------------------------------------------------------------------------------
	\subsection{Samples from  Functions}
	%---------------------------------------------------------------------------------------------------------------------
	Next, we approximate the  functions $f_i/\|f_i\|_{L^1}$, $i=1,2$, 
	where 
	$f_1\colon [0,1]^9\to\R$ 
	and
	$f_2\colon [0,1]^{10} \to\R$ 
	are given by
	\begin{align} 
		f_1(x) &= B_2(x_1)B_4(x_3)B_6(x_8)+B_2(x_2)B_4(x_5)B_6(x_6)+B_2(x_4)B_4(x_7)B_6(x_9),\label{num_2_1}\\
		f_2(x) &= 10\sin(\pi x_1 x_2)+20(x_3+0.5)^2+10x_4+ 5x_5. \label{num_2_2}
	\end{align}
	Here $B_2$, $B_4$ and $B_6$ are the $L_2$ normalized $B$-splines of order $2$, $4$ and $6$ supported on $[0,1]$.
	Note, that the spline function $f_1$ was also used in \cite{BPS2020} for numerical evaluations. The function $f_2$ is the so-called Friedmann-1 function.
	We use $d_s=3$ iterations within the heuristic of Section~\ref{sec:find_u} for $f_1$ and $d_s=2$ iterations for $f_2$.
	The results are given in Table \ref{tab_results_splines}.
	Figure~\ref{fig_components2} shows a diagram with the weights of the recovered couplings $u_k$.
	We observe that all recovered couplings $u_k$ with a significant weight match to the definitions of the functions $f_i$, $i=1,2$.

	\begin{table}
		\begin{center}
			\resizebox*{!}{1.6cm}{
				\begin{tabular}{c|cccc}
					Method&$\mathcal L_f(x^1,...,x^N)$&$\mathcal L_{\hat p}(x^1,...,x^N)$&$e_{L^1}(\hat p,f)$&$e_{L^2}(\hat p,f)$\\\hline
					full&$7009.8\pm 59.4$&$7026.7\pm79.2$&$0.0971\pm0.0053$&$0.0963\pm0.0056$\\
					diagonal&$7009.8\pm 59.4$&$7001.9\pm52.0$&$0.0828\pm0.0030$&$0.0828\pm0.0037$\\
					von Mises&$7009.8\pm59.4$&$6932.4\pm56.4$&$0.1292\pm0.0034$&$0.1272\pm0.0041$
				\end{tabular}
			}\\
			\vspace{.5cm}
			\resizebox*{!}{1.6cm}{
				\begin{tabular}{c|cccc}
					Method&$\mathcal L_f(x^1,...,x^N)$&$\mathcal L_{\hat p}(x^1,...,x^N)$&$e_{L^1}(\hat p,f)$&$e_{L^2}(\hat p,f)$\\\hline
					full&$630.0\pm 43.9$&$566.1\pm43.0$&$0.1150\pm0.0131$&$0.1419\pm0.0142$\\
					diagonal&$630.0\pm 43.9$&$558.6\pm53.8$&$0.1175\pm0.0135$&$0.1444\pm0.0139$\\
					von Mises&$630.0\pm43.9$&$549.8\pm56.4$&$0.1220\pm0.0124$&$0.1490\pm0.0128$
				\end{tabular}
			}
		\end{center}
		\caption{Approximation of the $f_i$ in \eqref{num_2_1} and \eqref{num_2_2} by the three sparse mixture models in Example \ref{ex:distr} for
			$N=10000$ samples. 
			Average value of the log likelihood function and relative $L_q$, $q=1,2$. Top: Spline function $f_1$, bottom: Friedmann-1 function $f_2$.}
		\label{tab_results_splines}
	\end{table}
	
	\begin{figure}
		\begin{subfigure}[t]{0.33\linewidth}
			\centering
			\includegraphics[width=\textwidth]{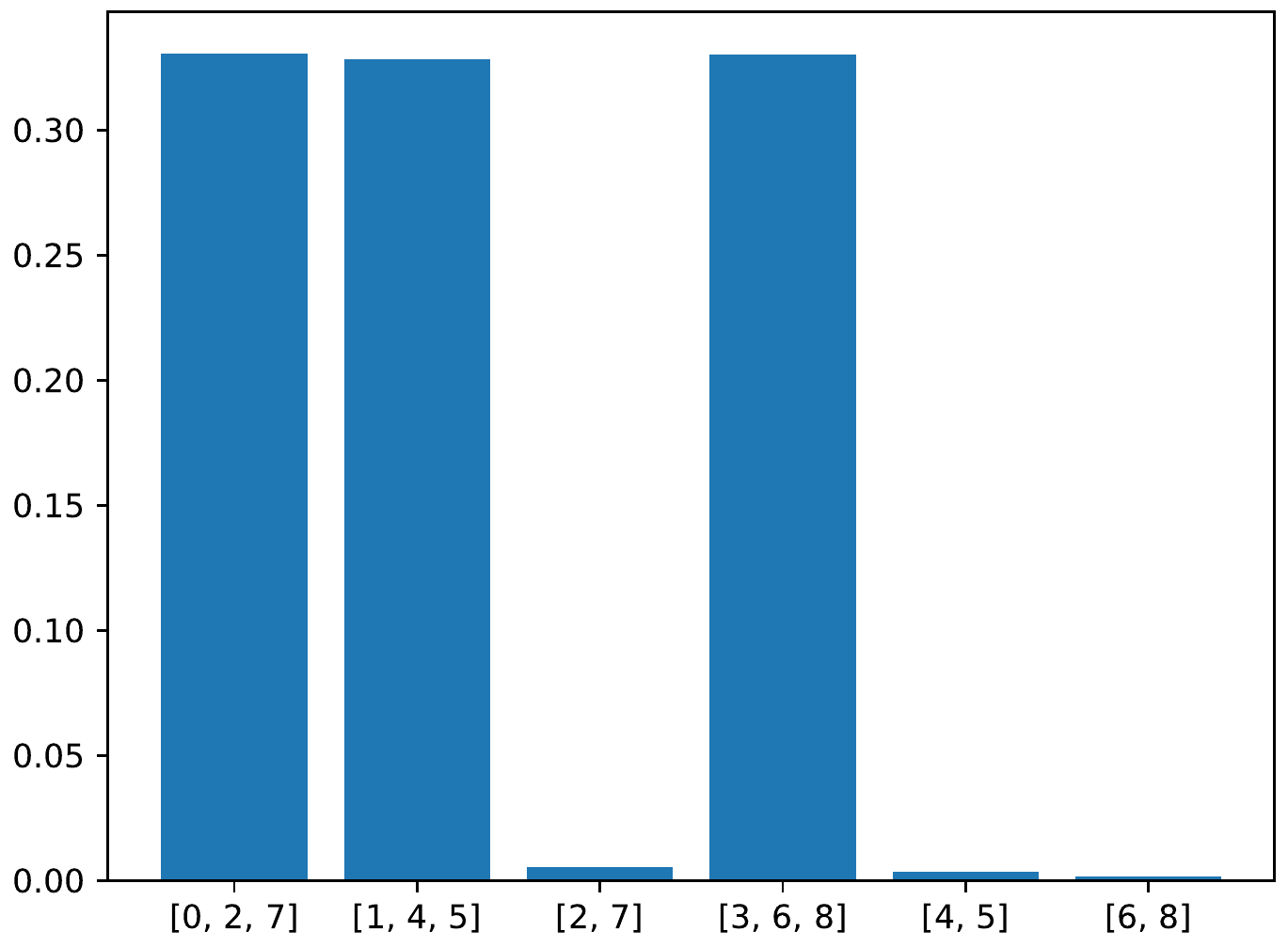}
		\end{subfigure}\hfill
		\begin{subfigure}[t]{0.33\linewidth}
			\centering
			\includegraphics[width=\textwidth]{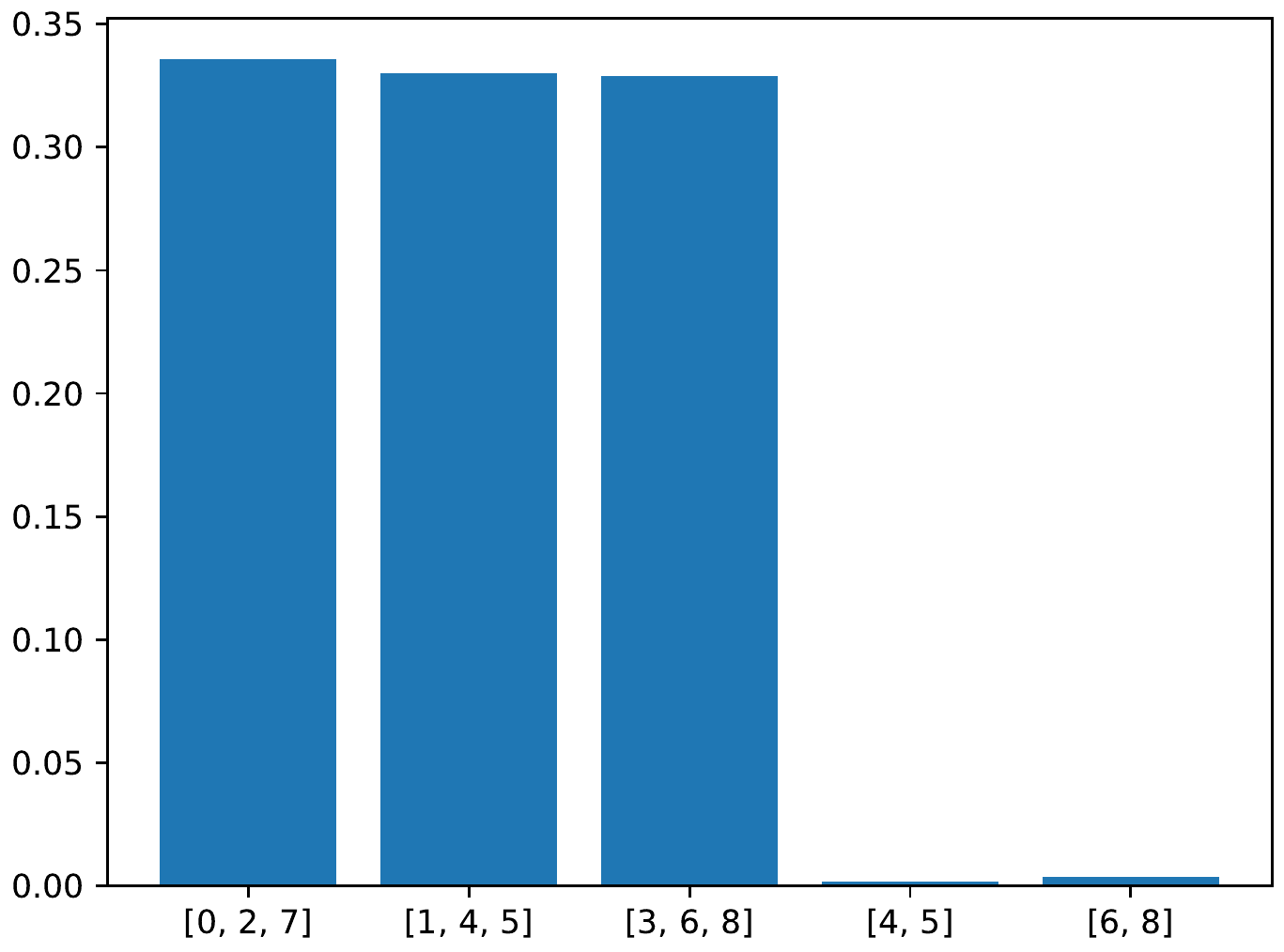}
		\end{subfigure}\hfill
		\begin{subfigure}[t]{0.33\linewidth}
			\centering
			\includegraphics[width=\textwidth]{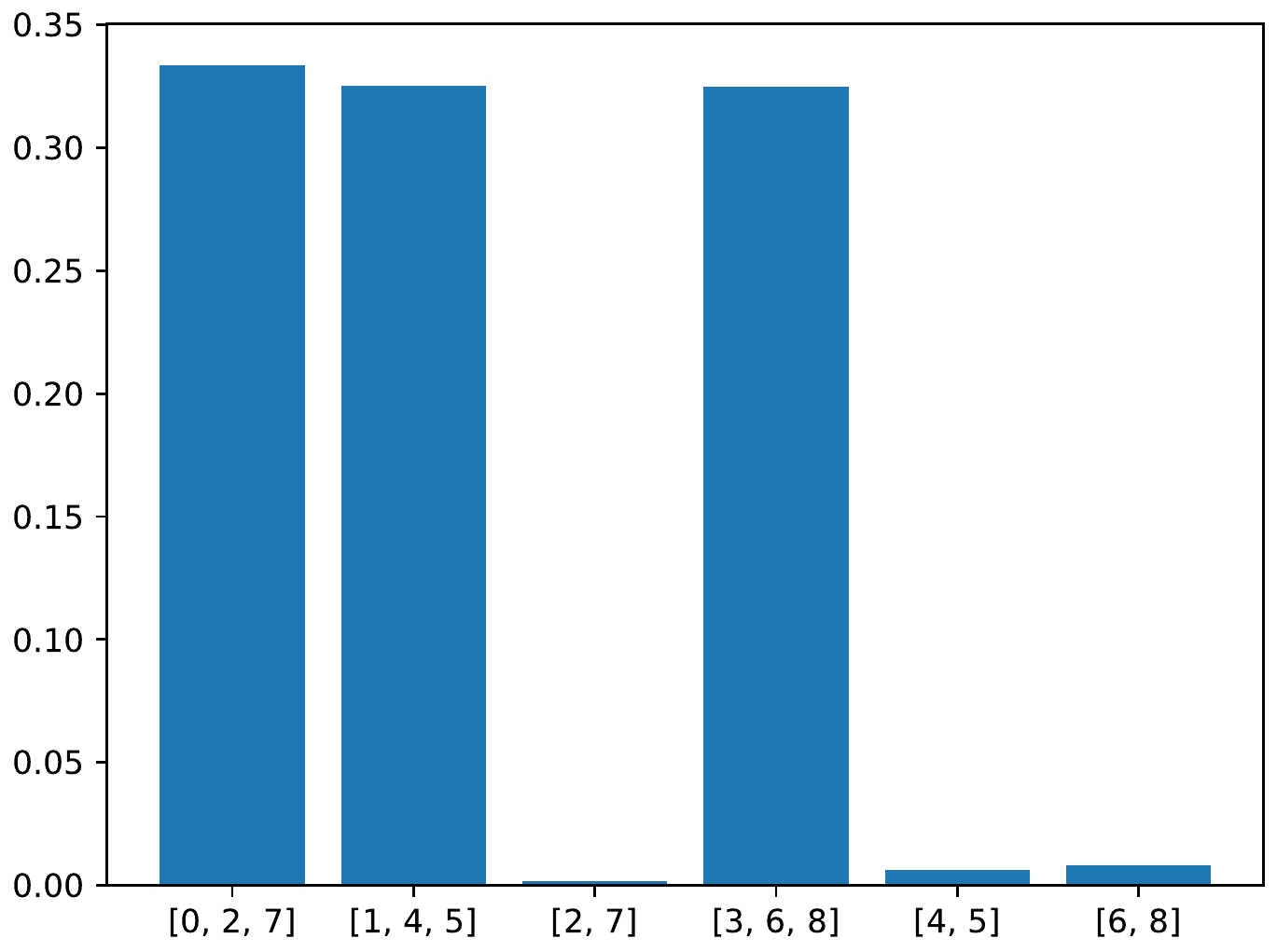}
		\end{subfigure}\hfill
		\begin{subfigure}[t]{0.33\linewidth}
			\centering
			\includegraphics[width=\textwidth]{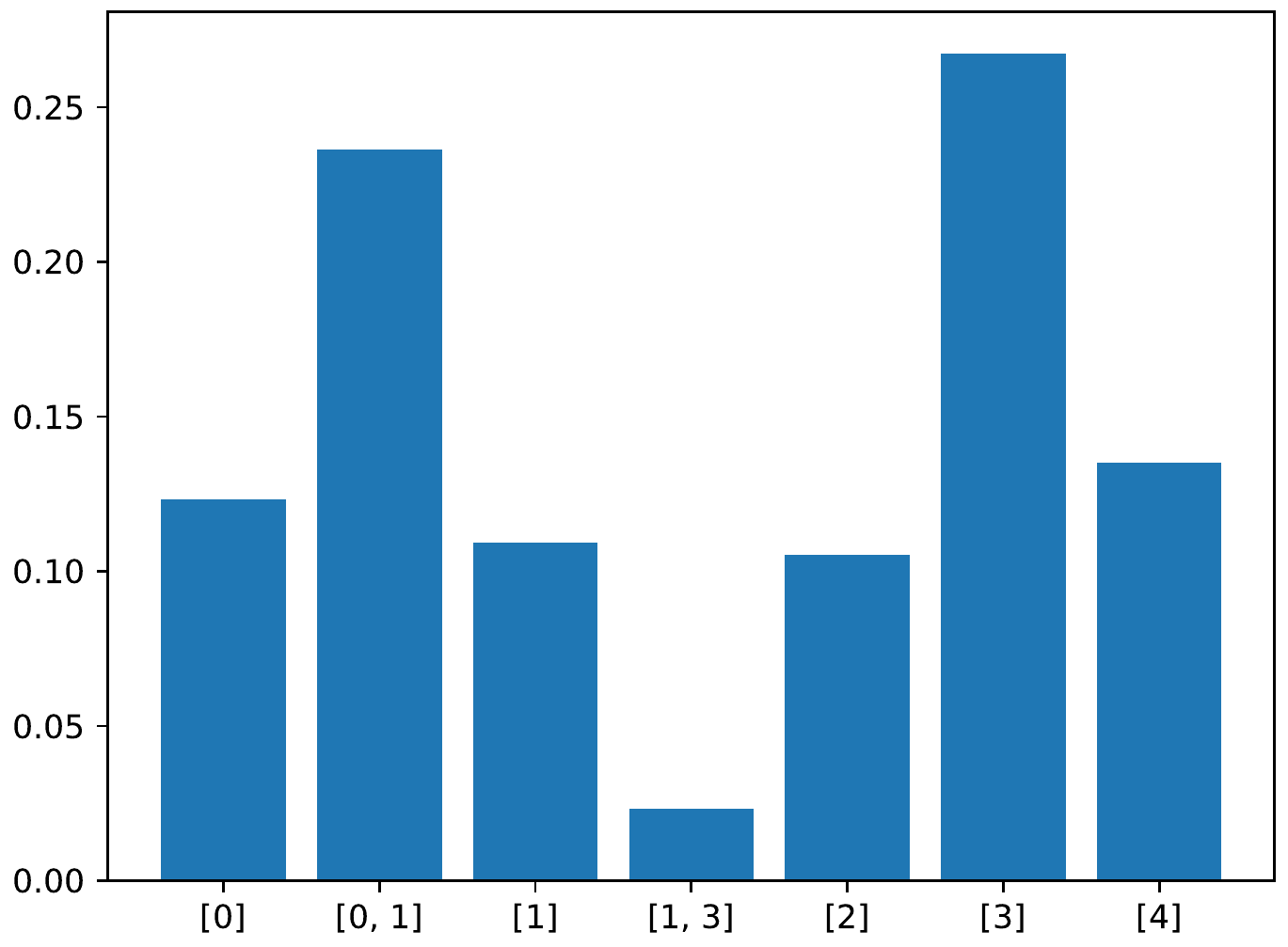}
			\caption*{wrapped}
		\end{subfigure}\hfill
		\begin{subfigure}[t]{0.33\linewidth}
			\centering
			\includegraphics[width=\textwidth]{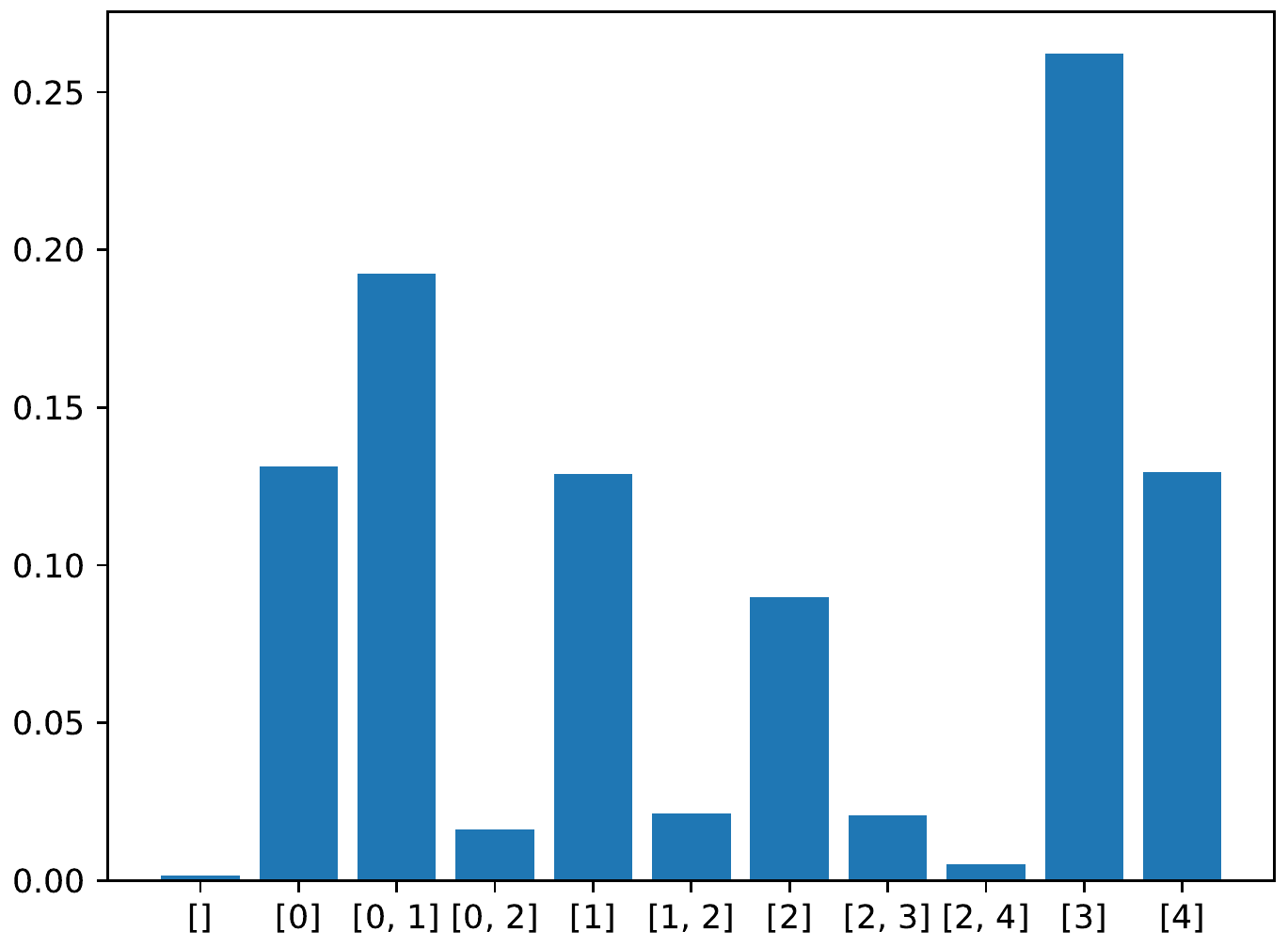}
			\caption*{comp.\ wrapped}
		\end{subfigure}\hfill
		\begin{subfigure}[t]{0.33\linewidth}
			\centering
			\includegraphics[width=\textwidth]{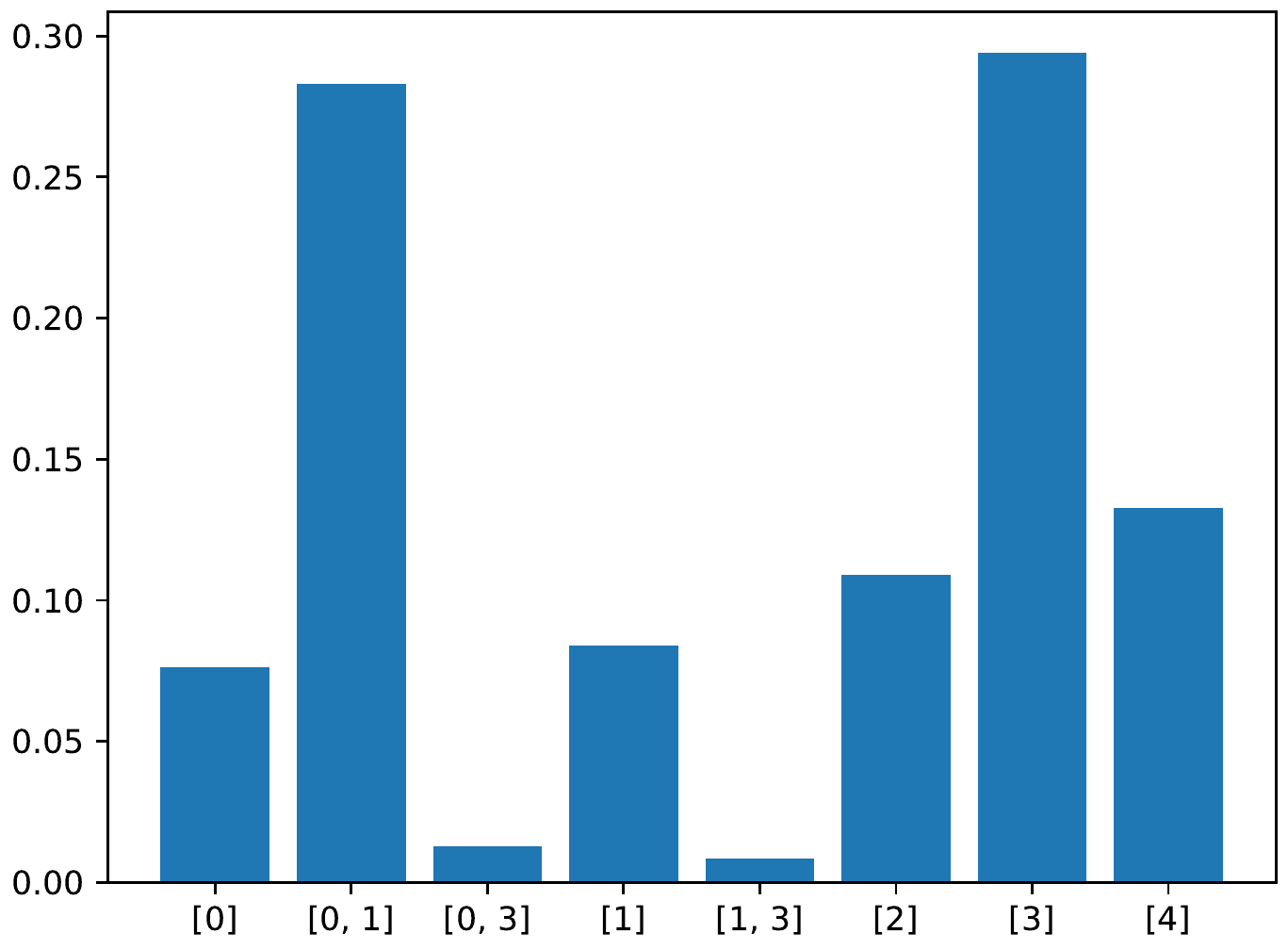}
			\caption*{von Mises}
		\end{subfigure}\hfill
		\caption{Weights of the recovered couplings $u_k$ for the approximation of the $f_i$ in \eqref{num_2_1} and \eqref{num_2_2} by the three sparse mixture models in Example~\ref{ex:distr} for $N=10000$ samples. Top: Spline function $f_1$, bottom: Friedmann-1 function $f_2$.}
		\label{fig_components2}
	\end{figure}

	The estimated parameters achieve a slightly worse log likelihood value than the original function $f$. 
	Thus, the original function fits the samples slightly better than the estimated sparse mixture model.
	This should not be surprising, since the function $f/\|f\|_{L_1}$ is not contained in the class of sparse mixture models. 
	Nevertheless, for the spline function $f_1$ the relative $L_1$ and $L_2$ errors are still comparable with the results from \cite{BPS2020}.

	%------------------------------------------------------------------
	\subsection{A Synthetic Image Example}
	%------------------------------------------------------------------
	
	In this example, we consider gray-valued images consisting of $7\times 10$ pixels which are sampled from five different classes.
	Each class contains noisy piece-wise constant images with one straight edge on a fixed position.
	In the following we learn a sparse mixture model for semi-supervised classification of images into one of these classes. 
	Here we assume, that not the whole image is given, but only the orientations of some of the gradients within the images.
	
	\paragraph{Image Generation}
	
	We generate an image $y=(y_{i,j})_{i,j}\in\R^{7\times10}$ using the following procedure:
	\begin{itemize}
		\item \textbf{Class 1:} We draw $a$ from $\mathcal N(0.1,0.05^2)$ and $b$ from $\mathcal N(0.9,0.1^2)$ and set $y_{i,j}=a$ for $j=1,2$ 
		and $y_{i,j}=b$ otherwise.
		\item \textbf{Class 2:} We draw $a$ from $\mathcal N(0.9,0.1^2)$ and $b$ from $\mathcal N(0.1,0.05^2)$ and set $y_{i,j}=a$ for $j=1,...,4$ and $y_{i,j}=b$ otherwise.
		\item \textbf{Class 3:} We draw $a$ from $\mathcal N(0.2,0.025^2)$ and $b$ from $\mathcal N(0.6,0.05^2)$ and set $y_{i,j}=a$ for $j=1,...,6$ and $y_{i,j}=b$ otherwise.
		\item \textbf{Class 4:} We draw $a$ from $\mathcal N(0.7,0.1^2)$ and $b$ from $\mathcal N(0.1,0.05^2)$ and set $y_{i,j}=a$ for $j=1,...,8$ and $y_{i,j}=b$ otherwise.
		\item \textbf{Class 5:} We draw $a$ from $\mathcal N(0.2,0.1^2)$ and $b$ from $\mathcal N(0.9,0.025^2)$ and set $y_{i,j}=a$ for $i=1,...,4$ and $y_{i,j}=b$ otherwise.
	\end{itemize}
	Finally, we add Gaussian white noise with standard deviation $0.2$ to each of the images.
	Figure \ref{fig_imgs} shows one sample from each class.
	
	\begin{figure}
		\centering
		\includegraphics[width=.19\linewidth]{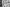}
		\includegraphics[width=.19\linewidth]{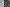}
		\includegraphics[width=.19\linewidth]{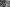}
		\includegraphics[width=.19\linewidth]{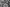}
		\includegraphics[width=.19\linewidth]{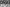}\vspace{.1cm}
		\includegraphics[width=.19\linewidth]{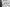}
		\includegraphics[width=.19\linewidth]{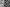}
		\includegraphics[width=.19\linewidth]{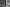}
		\includegraphics[width=.19\linewidth]{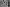}
		\includegraphics[width=.19\linewidth]{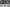}\vspace{.1cm}
		\includegraphics[width=.19\linewidth]{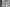}
		\includegraphics[width=.19\linewidth]{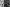}
		\includegraphics[width=.19\linewidth]{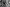}
		\includegraphics[width=.19\linewidth]{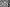}
		\includegraphics[width=.19\linewidth]{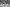}\vspace{.1cm}
		\caption{Three samples from each class.}
		\label{fig_imgs}
	\end{figure}
	
	\paragraph{Gradient Orientations and Data Generation}
	
	We assume that not the full images are given, but only the orientations of some of the gradients within the images.
	For an image $y=(y_{i,j})_{i,j}\in\R^{m\times n}$, 
	the central discrete gradient at position $(i,j)\in\{2,...,m-1\}\times\{2,...,n-1\}$ 
	is defined as the vector 
	$(y_{i+1,j}-y_{i-1,j},y_{i,j+1}-y_{i,j-1})$.
	Consequently, the orientation of the gradient at position 
	$(i,j)$ is given by 
	$$\omega_{i,j}\coloneqq \frac{1}{2\pi}\arctan^*(\tfrac{y_{i+1,j}-y_{i-1,j}}{y_{i,j+1}-y_{i,j-1}}),$$ 
	where $\arctan^*$ again denotes the quadrant specific inverse of the tangent as defined in \eqref{eq_atan2}.
	Finally, we assume that not all of the gradient orientations are given, but only the $\omega_{ij}$ with $(i,j)\in\mathcal I$, where
	$$
	\mathcal I\coloneqq \{(2,2),(2,3),(2,6),(2,7),(4,4),(4,5),(4,8),(4,9),(6,2),(6,3),(6,6),(6,7)\}.
	$$
	Figure~\ref{fig_pattern} visualizes the pixels corresponding to the positions $(i,j)\in\mathcal I$.

	\begin{figure}
		\centering
		\includegraphics[width=.5\linewidth]{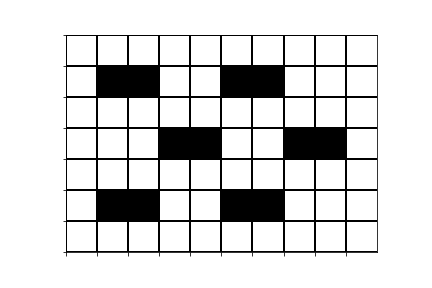}
		\caption{The black marked pixels represent the positions $(i,j)\in\mathcal I$.}
		\label{fig_pattern}
	\end{figure}
	
	Now, we generate $N_{\text{train}}=10000$ training samples and $N_{\text{test}}=1000$ test samples, where each sample $x\in\T^{12}$ is generated by the following steps: 
	First, we choose randomly a class $c\in\{1,...,5\}$ accordingly to some fixed (but for the classification unknown) probabilities $\alpha$.
	Second, we generate an image $y$ from class $c$ as described above and compute the gradient orientations $\omega_{ij}$.
	Finally, we define our sample $x\in\T^{12}$ as $x\coloneqq (\omega_{ij})_{(i,j)\in\mathcal I}$.  
	
	We visualize the components of $(\omega_{ij})_{(i,j)\in\mathcal I}$ 
	by histograms of $10000$ samples from class 2 in Figure~\ref{fig_histogram_data}.
	
	\begin{remark}
		Let us briefly comment why the above sampling generation fits into our model setting.
		If the $Y_{i,j}$ are i.i.d. Gaussian distributed (noise on constant areas), 
		then the components of their centered gradients are also
		i.i.d. Gaussian distributed for $(i,j) \in  \mathcal I$. 
		Note that the special set $\mathcal I$ is needed to ensure independence 
		of the random variables in the gradient.
		Finally, it follows from the transformation theorem, 
		that the random variable, where these $\omega_{i,j}$ are sampled from, 
		are uniformly distributed, see e.g. \cite{DLMM2002}.
		Thus, in our samples $x$ most of the components 
		will arise from a uniformly distributed random variable (constant areas)
		and only few ones (phases of the gradients at edges) 
		must approximated by a Gaussian-like mixture.
	\end{remark}
	
	\begin{figure}
		\centering
		\includegraphics[width=\linewidth]{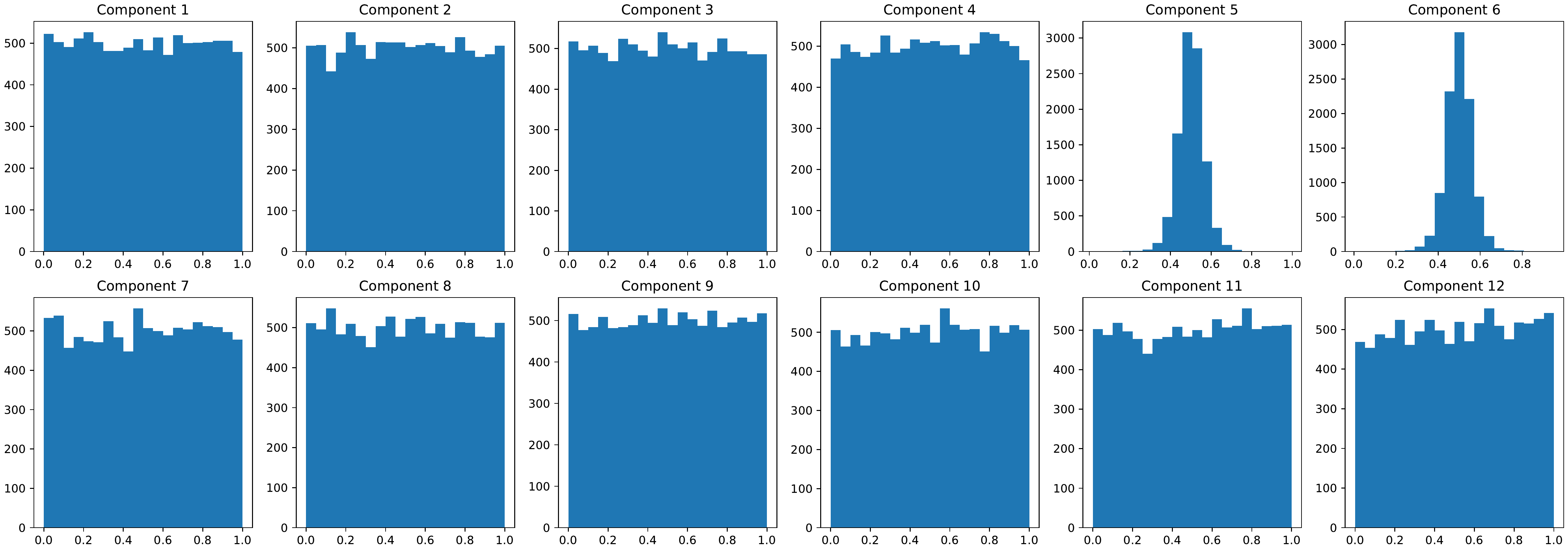}
		\caption{Histograms of the components of $(\omega_{ij})_{(i,j)\in\mathcal I}$ for $10000$ samples from class 2.}
		\label{fig_histogram_data}
	\end{figure}
	
	\paragraph{Semi-Supervised Classification}
	In the following, assume that we have given $N_{\text{train}}$ unlabeled training samples 
	$x^1,...,x^{N_\text{train}}$ (i.e.\ it is unknown, which sample belongs to which class).
	Additionally, we have given $3$ labeled samples $x^{1,c},...,x^{3,c}$ from each class $c=1,...,5$.
	Based on this, we would like to classify the samples from the test set into the five classes.
	For this, we perform three steps.
	\begin{itemize}
		\item First, we estimate the components $u_k$ and the parameters $(\alpha,\mu,\Sigma)$ of a sparse mixture model as described in 
		Section \ref{sec:spamm}, where we use $d_s=4$ steps within the heuristic from Section \ref{sec:find_u}.
		\item In general, this mixture model will have $K > 5$ classes 
		and we have to assign the components to the different classes. 
		For this we assume that we have additionally given three labeled samples 
		$x^{1,c},...,x^{3,c}$ from each class $c=1,...,5$.
		Then we assign $k$ to class $c(k) \in \{1,\ldots,5\}$ by
		$$
		c(k) \coloneqq \argmax_{c=1,...,5} \sum_{i=1}^3 \log \left( p_{u_k}(x^{i,c}_{u_k}|\vartheta_k) \right).
		$$
		\item Finally, we predict for a data point $x$ from the test set the appropriate class
		$$
		\hat c \coloneqq \argmax_{c=1,...,5}\sum_{k\in [K], c_k=c} \alpha_k p_{u_k}(x_{u_k}|\vartheta_k).
		$$                                                  
	\end{itemize}
	Using this procedure, we achieve an accuracy of $93.6 \%$ on the test set.

	%%%%%%%%%%%%%%%%%%%%%%%%%%%%%%%%%%%%%%%%%%%%%%%%%%%%%%%%%%%%%%%%%%%%%%%%%%%%%%%%%%%%%%%%%%%%%%%%%%%%%%%%%%%%%%%%%%%%%%%%%%%%%%
	\appendix
	%-----------------------------------------------------------
	\section{EM Algorithm for Distributions in Example \ref{ex:distr}  } \label{sec:A}
	%-----------------------------------------------------------
	We summarize the EM algorithms for the mixture models with components
	from Example \ref{ex:distr} i) - iii) in this order.

	\begin{algorithm}[!ht]
		\caption{EM Algorithm for MMs with Wrapped Normal Distribution}\label{alg_SPAMM}
		\begin{algorithmic}
			\State Input: $x=(x^1,...,x^N)\in\R^{d, N}$, 
			initialization $\alpha^{(0)}$, $\vartheta^{(0)} = (\mu^{(0)}, \Sigma^{(0)})$.
			\For {$r=0,1,...$}
			\State \textbf{E-Step:} For $k=1,...,K$, $i=1,\ldots,N$ and $l\in\Z^{|u|}$ compute 
			\begin{align}
				\beta_{i,k,l}^{(r)}
				&=
				\frac{\alpha_k^{(r)} \mathcal N(x^i_{u_k}+l|\mu_k^{(r)},\Sigma_k^{(r)})) }
				{\sum_{j=1}^K \alpha_j^{(r)} p(x^i_{u_j}|\mu_j^{(r)},\Sigma_j^{(r)})}
			\end{align}
			\State \textbf{M-Step:} For $k=1,...,K$ compute
			\begin{align}
				\alpha_k^{(r+1)}
				& = \frac1N \sum_{i=1}^N w_i\sum_{l\in\Z^{|u_k|}} \beta_{i,k,l}^{(r)},\\
				\mu_k^{(r+1)}&=\frac{1}{N\alpha_k^{(r+1)}}\sum_{i=1}^N w_i\sum_{l\in\Z^{|u_k|}}\beta_{i,k,l}^{(r)}(x_{u_k}^i+l),\\
				\Sigma_k^{(r+1)}&=\frac{1}{N\alpha_k^{(r+1)}}\sum_{i=1}^N w_i\sum_{l\in\Z^{|u_k|}}\beta_{i,k,l}^{(r)}(x_{u_k}^i+l-\mu_k^{(r+1)})(x_{u_k}^i+l-\mu_{k}^{(r+1)})^\tT.
			\end{align}
			\EndFor
		\end{algorithmic}
	\end{algorithm}
	
	%-------------------------------------------------------------------------------------------------------
	\begin{algorithm}[!ht]
		\caption{EM Algorithm for MMs with Diagonal Wrapped Normal Distribution}\label{alg_SPAMM_diag}
		\begin{algorithmic}
			\State Input: $x=(x^1,...,x^N)\in\R^{d, N}$, 
			initialization $\alpha^{(0)}$, $\vartheta^{(0)} = (\mu^{(0)}, (\sigma_{j,k}^{(0)})^2)$.
			\For {$r=0,1,...$}
			\State \textbf{E-Step:} For $k=1,...,K$, $i=1,\ldots,N$, $j\in u_k$ and $m\in\Z$ compute 
			\begin{align}
				\gamma^{(r)}_{i,k,m,j}
				&=
				\frac{\alpha_k^{(r)}\mathcal N(x_j^i+m|\mu_{j,k}^{(r)},(\sigma_{j,k}^{(r)})^2)\prod_{s\in u_k\backslash \{j\}}\mathcal N_w(x^i_s|\mu_{s,k}^{(r)},(\sigma_{s,k}^{(r)})^2)}{\sum_{t=1}^K\alpha_t^{(r)}\prod_{s\in u_t}\mathcal N_w(x^i_s|\mu^{(r)}_{s,t},(\sigma^{(r)}_{s,t})^2)}
			\end{align}
			\State \textbf{M-Step:} For $k=1,...,K$ compute
			\begin{align}
				\alpha_k^{(r+1)}
				& =\frac1N\sum_{i=1}^N w_i\sum_{m\in\Z}\gamma_{i,k,m,j}^{(r)},\quad\text{for any }j\in u_k\\
				\mu_{j,k}^{(r+1)}&=\frac{1}{N\alpha_k^{(r+1)}}\sum_{i=1}^N w_i\sum_{m\in\Z}\gamma_{i,k,m,j}^{(r)}(x^i_j+m),\\
				(\sigma_{j,k}^{(r+1)})^2&=\frac{1}{N\alpha_k^{(r+1)}}\sum_{i=1}^N w_i\sum_{m\in\Z}\gamma_{i,k,m,j}^{(r)}(x^i_j+m-\mu_{j,k}^{(r+1)})^2.
			\end{align}
			\EndFor
		\end{algorithmic}
	\end{algorithm}
	
	%----------------------------------------------------------------------------------------------------------------
	\begin{algorithm}[!ht]
		\caption{EM Algorithm for MMs with products of von Mises Distributions}\label{alg_SPAMM_vM}
		\begin{algorithmic}
			\State Input: $x=(x^1,...,x^N)\in\R^{d, N}$, 
			initialization $\alpha^{(0)}$, $\vartheta^{(0)} = (\mu_{j,k}^{(0)}, \kappa_{j,k}^{(0)})$.
			\For {$r=0,1,...$}
			\State \textbf{E-Step:} For $k=1,...,K$, $i=1,\ldots,N$ compute 
			\begin{align}
				\beta^{(r)}_{i,k}
				&=
				\frac{\alpha_k^{(r)}
					\prod_{j\in u_k}p_M(x^i_j|\mu_{j,k}^{(r)},\kappa_{j,k}^{(r)})}{\sum_{t=1}^K\alpha_t^{(r)}
					\prod_{j\in u_t}p_M(x^i_j|\mu^{(r)}_{j,t},\kappa^{(r)}_{j,t})}
			\end{align}
			\State \textbf{M-Step:} For $k=1,...,K$ and $j\in u_k$ compute
			\begin{align}
				\alpha_k^{(r+1)}
				=\frac1N\sum_{i=1}^N w_i\beta_{i,k}^{(r)},\quad
				\mu_{j,k}^{(r+1)}=\arctan^*\Big(\tfrac{S_{j,k}^{(r)}}{C_{j,k}^{(r)}}\Big),\quad \kappa_{j,k}^{(r+1)}=A^{-1}(R_{j,k}^{(r)}),
			\end{align}
			where
			\begin{align}
				C_{j,k}^{(r)}&=\sum_{i=1}^N w_i\beta_{ik}^{(r)}\cos(2\pi x^i_j),\quad S_{j,k}^{(r)}=\sum_{i=1}^N w_i\beta_{ik}^{(r)}\sin(2\pi x^i_j),\\ 
				R_{j,k}^{(r)}&=\tfrac{1}{N\alpha_k^{(r+1)}}\sqrt{(S_{j,k}^{(r)})^2+(C_{j,k}^{(r)})^2}.
			\end{align}
			\EndFor
		\end{algorithmic}
	\end{algorithm}

	%----------------------------------------------------------------------
	\section{The Weighted Kolmogorov-Smirnov Test} \label{sec:test}
	%----------------------------------------------------------------------
	We briefly review the weighted Kolmogorov-Smirnov (KS) test. 
	The following definition and facts about the KS test can be found in \cite{M2011}.
	Given univariate samples $x^1,...,x^N\in[0,1]$ and a probability distribution defined 
	by its cumulative distribution function $f\colon[0,1]\to[0,1]$, we test the hypothesis
	$$H_0\colon (x^i)_i \text{ belong to the distribution } f$$
	against the alternative 
	$$H_1\colon (x^i)_i \text{ belong not to the distribution } f.$$
	We define the empirical cumulative density function $f_N\colon[0,1]\to[0,1]$ of the samples $(x^i)_i$ by
	$$
	f_N=\frac1N\sum_{i=1}^N 1_{[x^i,1]}.
	$$
	Then, the hypothesis $H_0$ is accepted, if the test statistic
	\begin{align}
		\mathrm{KS}((x^i)_i)\coloneqq \sqrt{N}\|f_N-f\|_{L^\infty}\label{eq_KS_unweighted}
	\end{align}
	is smaller or equal than some constant $c$, which was fixed a priori and controls the significance level of the test.
	It is shown that for $X^1,...,X^N$ i.i.d.\ random variables with cumulative distribution function $f$, 
	the KS test statistic $\mathrm{KS}((X^i)_i)$ converges in distribution to the Kolmogorov distribution as $N\to\infty$.
	
	The test can be extended for weighted samples $(w_1,x^1),...,(w_N,x^N)\in\R_{>0}\times[0,1]$ 
	by replacing the empirical cumulative distribution function by
	$$
	f_N=\frac{1}{\sum_{i=1}^N w_i}\sum_{i=1}^N w_i 1_{[x^i,1]}.
	$$
	Further, one has to replace $\sqrt{N}$ in \eqref{eq_KS_unweighted}. Here, \cite{M2011} 
	suggests to replace $N$ by $\frac{\big(\sum_{i=1}^N w_i)\big)^2}{\sum_{i=1}^N w_i^2}$. 
	Thus, the weighted KS test statistic reads as
	\begin{align}
		\mathrm{KS}((w_i,x^i)_i)=\sqrt{\frac{\big(\sum_{i=1}^N w_i)\big)^2}{\sum_{i=1}^N w_i^2}}\|f_N-f\|_{L^\infty}.\label{eq_KS_weighted}
	\end{align}
	
	\begin{remark}
		Note that for two cumulative distribution functions $f$ and $g$ the term $d(f,g)=\|f-g\|_{L^\infty}$ 
		defines a metric on the probability measures on $[0,1]$.
		In particular, for fixed weights $w_i$, 
		the weighted KS test statistic can be interpreted as the distance of the measure induced by $f$ 
		to the measure $\frac{1}{\sum_{i=1}^N w_i}\sum_{i=1}^N w_i \delta_{x^i}$, 
		where $\delta_{x}$ is the Dirac-measure in $x$.
	\end{remark}
	
	\begin{remark}
		For the uniform distribution the weighted KS test statistic can be easily computed. Assume that the $x^i$ are sorted, i.e.\ $x^1\leq\cdots\leq x^N$, and denote by $s_i=\frac{\sum_{j=1}^i w_j}{\sum_{j=1}^N w_j}$. Then the \eqref{eq_KS_weighted} is given by
		$$
		\mathrm{KS}((w_i,x^i)_i)=\sqrt{\frac{\big(\sum_{i=1}^N w_i)\big)^2}{\sum_{i=1}^N w_i^2}}\max_{i=1,...,N}\{\max(s_i - x_i,x_i-s_{i-1})\}.
		$$
	\end{remark}
	
	%\bibliographystyle{plain}
	%\bibliography{lit}
	%\printbibliography
	%\nocite{*}
	%\addbibresource{mybib.bib}
	%\nocite{*}

	%-----------------------------------------------------------
	\section*{Acknowledgment}
	Funding by the BMBF 01|S20053B project SA$\ell$E
	and by the German Research Foundation (DFG) with\-in the project STE 571/16-1 SUPREMATIM
	is gratefully acknowledged.
	
	%\bibliographystyle{abbrv}
	%\bibliography{ref}

\begin{thebibliography}{10}
		
		\bibitem{AS2009}
		Y.~Agiomyrgiannakis and Y.~Stylianou.
		\newblock Wrapped {G}aussian mixture models for modeling and high-rate
		quantization of phase data of speech.
		\newblock {\em IEEE Transactions on Audio, Speech, and Language Processing},
		17(4):775--786, 2009.
		
		\bibitem{ADDJ2003}
		C.~Andrieu, N.~De~Freitas, A.~Doucet, and M.~I. Jordan.
		\newblock An introduction to {MCMC} for machine learning.
		\newblock {\em Machine learning}, 50(1):5--43, 2003.
		
		\bibitem{Banerjee2005ClusteringOT}
		A.~Banerjee, I.~Dhillon, J.~Ghosh, and S.~Sra.
		\newblock Clustering on the unit hypersphere using von mises-fisher
		distributions.
		\newblock {\em Journal of Machine Learning Research}, 6:1345--1382, 09 2005.
		
		\bibitem{BPS2020}
		F.~Bartel, D.~Potts, and M.~Schmischke.
		\newblock Grouped transformations in high-dimensional explainable {ANOVA}
		approximation.
		\newblock {\em arXiv preprint arXiv:2010.10199}, 2020.
		
		\bibitem{BGG2009}
		G.~Beylkin, J.~Garcke, and M.~J. Mohlenkamp.
		\newblock Multivariate regression and machine learning with sums of separable
		functions.
		\newblock {\em SIAM Journal on Scientific Computing}, 31(3):1840--1857, 2009.
		
		\bibitem{BDL2011}
		P.~Binev, W.~Dahmen, and P.~Lamby.
		\newblock Fast high-dimensional approximation with sparse occupancy trees.
		\newblock {\em Journal of Computational and Applied Mathematics},
		235(8):2063--2076, 2011.
		
		\bibitem{BGS2006}
		C.~Bouveyron, S.~Girard, and C.~Schmid.
		\newblock High-dimensional data clustering.
		\newblock {\em Computational Statistics \& Data Analysis}, 52(1):502--519,
		2007.
		
		\bibitem{B1989}
		J.~Breckling.
		\newblock {\em The analysis of directional time series: applications to wind
			speed and direction}, volume~61 of {\em Lecture Notes in Statistics}.
		\newblock Springer-Verlag, Berlin, 1989.
		
		\bibitem{Byrne2017}
		C.~L. Byrne.
		\newblock {\em The {EM} Algorithm: Theory, Applications and Related Methods}.
		\newblock Lecture Notes, University of Massachusetts, 2017.
		
		\bibitem{CMO1997}
		R.~Caflisch, W.~Morokoff, and A.~Owen.
		\newblock Valuation of mortgage-backed securities using {B}rownian bridges to
		reduce effective dimension.
		\newblock {\em Journal of Computational Finance}, 1:27--46, 1997.
		
		\bibitem{CJJ2012}
		R.~Chitta, R.~Jin, and A.~K. Jain.
		\newblock Efficient kernel clustering using random {F}ourier features.
		\newblock {\em 2012 IEEE 12th International Conference on Data Mining}, pages
		161--170, 2012.
		
		\bibitem{CH2000}
		S.~Chr{\'e}tien and A.~O. Hero.
		\newblock Kullback proximal algorithms for maximum-likelihood estimation.
		\newblock {\em IEEE Transactions on Information Theory}, 46(5):1800--1810,
		2000.
		
		\bibitem{CH2008}
		S.~Chr{\'e}tien and A.~O. Hero.
		\newblock On {EM} algorithms and their proximal generalizations.
		\newblock {\em ESAIM: Probability and Statistics}, 12:308--326, 2008.
		
		\bibitem{CDW2014}
		P.~Constantine, E.~Dow, and Q.~Wang.
		\newblock Active subspace methods in theory and practice: Applications to
		kriging surfaces.
		\newblock {\em SIAM Journal on Scientific Computing}, 36, 2014.
		
		\bibitem{CEHW2017}
		P.~Constantine, A.~Eftekhari, J.~Hokanson, and R.~Ward.
		\newblock A near-stationary subspace for ridge approximation.
		\newblock {\em Computer Methods in Applied Mechanics and Engineering},
		326:402--421, 2017.
		
		\bibitem{DD2020}
		J.~Delon and A.~Desolneux.
		\newblock A {W}asserstein-type distance in the space of {G}aussian mixture
		models.
		\newblock {\em SIAM Journal on Imaging Sciences}, 13(2):936—970, 2020.
		
		\bibitem{DLR1977}
		A.~P. Dempster, N.~M. Laird, and D.~B. Rubin.
		\newblock Maximum likelihood from incomplete data via the {EM} algorithm.
		\newblock {\em Journal of the Royal Statistical Society. Series B
			(Methodological)}, 39(1):1--38, 1977.
		
		\bibitem{DLMM2002}
		A.~Desolneux, L.~S., L.~Moisan, and J.-M. Morel.
		\newblock Dequantizing image orientation.
		\newblock {\em IEEE Transactions on Image Processing}, 11(10):1129--1140, 2002.
		
		\bibitem{DPW2011}
		R.~DeVore, G.~Petrova, and P.~Wojtaszczyk.
		\newblock Approximation of functions of few variables in high dimensions.
		\newblock {\em Constructive Approximation}, 33:125--143, 2011.
		
		\bibitem{F1987}
		N.~Fisher.
		\newblock Problems with the current definitions of the standard deviation of
		wind direction.
		\newblock {\em Journal of Applied Meteorology and Climatology},
		26(11):1522--1529, 1987.
		
		\bibitem{FSV2012}
		M.~Fornasier, K.~Schnass, and J.~Vyb{\'i}ral.
		\newblock Learning functions of few arbitrary linear parameters in high
		dimensions.
		\newblock {\em Foundations of Computational Mathematics}, 12:229--262, 2012.
		
		\bibitem{FHW1981}
		M.~D. Fraser, Y.-S. Hsu, and J.~J. Walker.
		\newblock Identifiability of finite mixtures of von {M}ises distributions.
		\newblock {\em The Annals of Statistics}, 9(5):1130--1131, 1981.
		
		\bibitem{GPT2020}
		M.~Goyal, M.~Pandey, and R.~Thakur.
		\newblock Exploratory analysis of machine learning techniques to predict energy
		efficiency in buildings.
		\newblock {\em 2020 8th International Conference on Reliability, Infocom
			Technologies and Optimization (Trends and Future Directions) (ICRITO)}, pages
		1033--1037, 2020.
		
		\bibitem{GBKT2020}
		R.~Griboval, G.~Balnchard, N.~Keriven, and Y.~Traonmilin.
		\newblock Compressive statistical learning with random feature moments.
		\newblock {\em arXiv preprint arXiv:1706.07180v3}, 2020.
		
		\bibitem{Gu2013}
		C.~Gu.
		\newblock {\em Smoothing spline {ANOVA} models}, volume 297 of {\em Springer
			Series in Statistics}.
		\newblock Springer, New York, second edition, 2013.
		
		\bibitem{HSSTTW2021}
		A.~Hashemi, H.~Schaeffer, R.~Shi, U.~Topcu, G.~Tran, and R.~Ward.
		\newblock Function approximation via sparse random features.
		\newblock {\em arXiv preprint arXiv:2103.03191}, 2021.
		
		\bibitem{HNABBSS2020}
		J.~Hertrich, D.~P.~L. Nguyen, J.-F. Aujol, D.~Bernard, Y.~Berthoumieu,
		A.~Saadaldin, and G.~Steidl.
		\newblock {PCA} reduced {G}aussian mixture models with applications in
		superresolution.
		\newblock {\em arXiv preprint arXiv:2009.07520}, 2020.
		
		\bibitem{Holtz2011}
		M.~Holtz.
		\newblock Sparse grid quadrature in high dimensions with applications in
		finance and insurance.
		\newblock In {\em Lecture Notes in Computational Science and Engineering},
		volume~77. Springer, 2011.
		
		\bibitem{HMS2004}
		H.~Holzmann, A.~Munk, and B.~Stratmann.
		\newblock Identifiability of finite mixtures - with applications to circular
		distributions.
		\newblock {\em The Indian Journal of Statistics}, 66:440--449, 2004.
		
		\bibitem{HBD2018}
		A.~Houdard, C.~Bouveyron, and J.~Delon.
		\newblock High-dimensional mixture models for unsupervised image denoising
		{(HDMI)}.
		\newblock {\em SIAM Journal on Imaging Sciences}, 11(4):2815--2846, 2018.
		
		\bibitem{JS2001}
		S.~R. Jammalamadaka and A.~SenGupta.
		\newblock {\em Topics in circular statistics}, volume~5 of {\em Series on
			Multivariate Analysis}.
		\newblock World Scientific Publishing Co., Inc., River Edge, NJ, 2001.
		
		\bibitem{K1974}
		D.~G. Kendall.
		\newblock Pole-seeking {B}rownian motion and bird navigation.
		\newblock {\em Journal of the Royal Statistical Society: Series B
			(Methodological)}, 36(3):365--402, 1974.
		
		\bibitem{KM2015}
		Y.~Kokkinos and K.~Margaritis.
		\newblock Multithreaded local learning regularization neural networks for
		regression tasks.
		\newblock In {\em EANN}, page 129—138. Springer International Publishing,
		2015.
		
		\bibitem{KSWW2010}
		F.~Y. Kuo, I.~H. Sloan, G.~W. Wasilkowski, and H.~Wo\'{z}niakowski.
		\newblock On decompositions of multivariate functions.
		\newblock {\em Mathematics of Computation}, 79(270):953--966, 2010.
		
		\bibitem{KGH2014}
		G.~Kurz, I.~Gilitschenski, and U.~D. Hanebeck.
		\newblock Efficient evaluation of the probability density function of a wrapped
		normal distribution.
		\newblock In {\em 2014 Sensor Data Fusion: Trends, Solutions, Applications
			(SDF)}, pages 1--5. IEEE, 2014.
		
		\bibitem{Laus2019}
		F.~Laus.
		\newblock {\em Statistical Analysis and Optimal Transport for Euclidean and
			Manifold-Valued Data}.
		\newblock PhD Thesis, TU Kaiserslautern, 2019.
		
		\bibitem{LTOS2019}
		Z.~Li, J.-F. Ton, D.~Oglic, and D.~Sejdinovic.
		\newblock Towards a unified analysis of random {F}ourier features.
		\newblock In {\em International Conference on Machine Learning}, pages
		3905--3914, 2019.
		
		\bibitem{LO2006}
		R.~Liu and A.~Owen.
		\newblock Estimating mean dimensionality of analysis of variance
		decompositions.
		\newblock {\em Journal of the American Statistical Association}, 101:712 --
		721, 2006.
		
		\bibitem{mardia2009directional}
		K.~Mardia and P.~Jupp.
		\newblock {\em Directional Statistics}.
		\newblock Wiley Series in Probability and Statistics. Wiley, 2009.
		
		\bibitem{M2010}
		K.~V. Mardia.
		\newblock Bayesian analysis for bivariate von {M}ises distributions.
		\newblock {\em Journal of Applied Statistics}, 37(3):515--528, 2010.
		
		\bibitem{MHTS2008}
		K.~V. Mardia, G.~Hughes, C.~C. Taylor, and H.~Singh.
		\newblock A multivariate von {M}ises distribution with applications to
		bioinformatics.
		\newblock {\em Canadian Journal of Statistics}, 36(1):99--109, 2008.
		
		\bibitem{MTS2007}
		K.~V. Mardia, C.~C. Taylor, and G.~K. Subramaniam.
		\newblock Protein bioinformatics and mixtures of bivariate von {M}ises
		distributions for angular data.
		\newblock {\em Biometrics}, 63(2):505--512, 2007.
		
		\bibitem{MP2000}
		G.~McLachlan and D.~Peel.
		\newblock {\em Finite Mixture Models}.
		\newblock Wiley Series in Probability and Statistics. Wiley, 2004.
		
		\bibitem{MLH2003}
		D.~Meyer, F.~Leisch, and K.~Hornik.
		\newblock The support vector machine under test.
		\newblock {\em Neurocomputing}, 55(1):169--186, 2003.
		
		\bibitem{M2011}
		J.~F. Monahan.
		\newblock {\em Numerical methods of statistics}, volume~7 of {\em Cambridge
			Series in Statistical and Probabilistic Mathematics}.
		\newblock Cambridge University Press, Cambridge, second edition, 2011.
		
		\bibitem{PS2021b}
		D.~Potts and M.~Schmischke.
		\newblock Interpretable approximation of high-dimensional data.
		\newblock {\em arXiv preprint arXiv:2103.13787}, 2021.
		
		\bibitem{PS2021a}
		D.~Potts and M.~Schmischke.
		\newblock Learning high-dimensional periodic functions with {F}ourier based
		methods.
		\newblock {\em arXiv preprint arXiv:1907.11412}, 2021.
		
		\bibitem{RR2008}
		A.~Rahimi and B.~Recht.
		\newblock Random features for large-scale kernel machines.
		\newblock In {\em Advances in Neural Information Processing Systems},
		volume~20, 2008.
		
		\bibitem{RC2013}
		C.~P. Robert and G.~Casella.
		\newblock {\em Monte {C}arlo statistical methods}.
		\newblock Springer Texts in Statistics. Springer-Verlag, New York, second
		edition, 2004.
		
		\bibitem{STA2021}
		H.~Shi, Y.~Traonmilin, and J.-F. Aujol.
		\newblock Sketched learning for image denoising.
		\newblock In A.~Elmoataz, J.~Fadili, Y.~Quéau, J.~Rabin, and L.~Simon,
		editors, {\em Scale Space and Variational Methods}, volume 12679 of {\em
			Lecture Notes in Computer Science}, pages 281--293. Springer, 2021.
		
		\bibitem{SB2005}
		P.~Smaragdis and P.~Boufounos.
		\newblock Learning source trajectories using wrapped-phase hidden {M}arkov
		models.
		\newblock In {\em IEEE Workshop on Applications of Signal Processing to Audio
			and Acoustics, 2005.}, pages 114--117. IEEE, 2005.
		
		\bibitem{Teicher1961}
		H.~Teicher.
		\newblock Identifiability of mixtures.
		\newblock {\em The Annals of Mathematical Statistics}, 32(1):244--248, 1961.
		
		\bibitem{Teicher1967}
		H.~Teicher.
		\newblock Identifiability of mixtures of product measures.
		\newblock {\em The Annals of Mathematical Statistics}, 38(4):1300--1302, 1967.
		
		\bibitem{TB1999}
		M.~E. Tipping and C.~M. Bishop.
		\newblock Mixtures of probabilistic principal component analyzers.
		\newblock {\em Neural Computation}, 11(2):443--482, 1999.
		
		\bibitem{WH2011}
		C.~F.~J. Wu and M.~S. Hamada.
		\newblock {\em Experiments: planning, analysis, and optimization}.
		\newblock Wiley Series in Probability and Statistics. John Wiley \& Sons, Inc.,
		Hoboken, NJ, second edition, 2009.
		
		\bibitem{YS1968}
		S.~J. Yakowitz and J.~D. Spragins.
		\newblock On the identifiability of finite mixtures.
		\newblock {\em The Annals of Mathematical Statistics}, pages 209--214, 1968.
		
		\bibitem{YLMJZ2012}
		T.~Yang, Y.-F. Li, M.~Mahdavi, R.~Jin, and Z.~Zhou.
		\newblock Nystr{\"o}m method vs random {F}ourier features: A theoretical and
		empirical comparison.
		\newblock {\em Advances in Neural Information Processing Systems}, 25:476--484,
		2012.
		
	\end{thebibliography}
	
\end{document}